\documentclass[a4paper,11pt,oneside,reqno]{amsart}
\usepackage{style}
\title{Well-posedness of McKean-Vlasov SDEs with density-dependent drift}
\author[Le]{Anh-Dung Le}
\address{Toulouse School of Economics, 1 Esplanade de l'Université, 31000 Toulouse, France}
\email[Le]{leanhdung1994@gmail.com}
\author[Villeneuve]{Stéphane Villeneuve}
\address{Toulouse School of Economics, 1 Esplanade de l'Université, 31000 Toulouse, France}
\email[Villeneuve]{stephane.villeneuve@tse-fr.eu}
\date{\today}
\thanks{This work was supported by ANR MaSDOL (No. ANR-19-CE23-0017); Air Force Office of Scientific Research, Air Force Material Command, USAF (No. FA9550-19-7026); and ANR Chess (No. ANR-17-EURE-0010).}

\begin{document}

\begin{abstract}
In this paper, we study well-posedness of McKean-Vlasov stochastic differential equations (SDE) whose drift depends pointwisely on marginal density and satisfies a local integrability condition in time-space variables. The drift and noise coefficients are assumed to be Lipschitz continuous in distribution variable with respect to Wasserstein metric $W_p$. Our approach is by approximation with mollifiers. We prove strong existence of a solution. Weak and strong uniqueness are obtained when $p=1$, the drift coefficient is bounded, and the diffusion coefficient is distribution free.
\end{abstract}

\maketitle



\tableofcontents

\section{Introduction} \label{introduction}
The study of distribution-dependent SDEs started with McKean’s seminal work \cite{mckean1966class} about Vlasov equation of plasma which had been proposed in \cite{kac1956foundations}. Classical results about the solvability of McKean-Vlasov equations include \cite{funaki1984certain,sznitman1984nonlinear,scheutzow1987uniqueness}. Since then, the literature on the well-posedness of McKean-Vlasov SDEs has been extended significantly. For recent results, we refer to \cite{mishura2020existence,chaudru_de_raynal_strong_2020,huang_mckean-vlasov_2021,10.3150/20-BEJ1268,huang2022singular,chaudru_de_raynal_well-posedness_2022}, the survey \cite{huang_distribution_2021}, and references therein. For applications on mean-field games, see the two-volume monograph \cite{carmona2018probabilistic}.

In this paper, we consider the density-dependent McKean-Vlasov SDEs studied in \cite{barbu2020nonlinear}. More precisely, let $p \in [1, \infty)$ and $\sP_p (\bR^d)$ be the space of Borel probability measures on $\bR^d$ with finite $p$-th moment. We endow $\sP_p (\bR^d)$ with the Wasserstein metric $W_p$. Let $T>0$ and $\bT$ be the interval $[0, T]$. We consider measurable functions
\begin{align}
b &: \bT \times \bR^d \times \bR_+ \times \sP_p (\bR^d) \to \bR^d, \\
\sigma &: \bT \times \bR^d \times \sP_p (\bR^d) \to \bR^d \otimes \bR^m.
\end{align}

Let $(B_t)$ be a $m$-dimensional Brownian motion ($m$-BM) and $(\cF_t)$ an admissible filtration (AF) on a probability space (PS) $(\Omega, \cA, \bP)$. Our object of study is the SDE
\begin{equation} \label{main_eq1}
\begin{cases} 
\diff X_t = b(t, X_t, \ell_t (X_t), \mu_t) \diff t + \sigma (t, X_t, \mu_t) \diff B_t , \\
\text{$\nu$ is the distribution of $X_0$, $\mu_t$ is that of $X_t$} , \\
\text{and $\ell_{t}$ is the density of $X_t$} .
\end{cases} 
\end{equation}

In the paper, we always mean by ``density'' the probability density function. First, we give a brief survey of the related literature. In \cite{hao_second_2023,issoglio_mckean_2023,issoglio_degenerate_2024}, $\sigma$ is constant and $b$ belongs to a Besov space. On the other hand, \cite{blanchard_probabilistic_2010,barbu_probabilistic_2011,belaribi_uniqueness_2012,belaribi_probabilistic_2013} studied generalized porous media equations where $b = 0$ and $\sigma$ is density-dependent but time-homogeneous. In \cite{jourdain_propagation_1998,jourdain_propagation_2013,cohen_wellposedness_2019,barbu2018probabilistic,barbu2020nonlinear,barbu2021solutions,barbu2021uniqueness,barbu_evolution_2023,barbu_uniqueness_2023,rehmeier_weighted_2023,grube_strong_2023}, $b$ and $\sigma$ are density-dependent but time-homogeneous. More generally, \cite{grube_strong_2024,barbu2023nonlinear} allow $b$ and $\sigma$ to be depend on time. However, $b$ and $\sigma$ are distribution-independent in those previously mentioned papers. For a recent survey on existence result of density-dependent SDEs, we refer to \cite{izydorczyk2019mckean}.

Our paper is closely related to \cite{hao_strong_2024,wang2023singular} where $b$ is both distribution-dependent and density-dependent. In particular, \cite{hao_strong_2024} used the total variation metric on $\sP (\bR^d)$ while \cite{wang2023singular} used supremum norm on the space of bounded densities. The current paper contributes in three aspects:
\begin{enumerate}
\item We use a mollifying argument whereas \cite{hao_strong_2024,wang2023singular} employed a Picard-iteration argument. We argue that our approach is more flexible because it can be extended to those situations where reasonable Hölder estimates are available. An interesting case for future study is when $b$ is allowed to grow linearly in space as in \cite{menozzi2021density}.
\item For existence result, the conditions on $\sigma$ in \cite{wang2023singular,hao_strong_2024} are more restrictive than ours. First, \cite{hao_strong_2024} assumed that $\sigma(t, x, \mu) = \sigma(t, x)$. Second, \cite{wang2023singular} assumed that $\sigma$ is Lipschitz in space, that $\nabla \sigma$ is Hölder continuous, and that $\|\sigma(t, \cdot, \mu) - \sigma(t, \cdot, \nu)\|_{C_b^{\alpha}} \lesssim \|\ell_\mu - \ell_\nu\|_{\infty}$. Here $\ell_\mu, \ell_\nu$ are the densities of $\mu, \nu$ respectively.
\item To be more aligned with existing literature of Mckean-Vlasov SDEs, we use Wasserstein metric for assumptions of continuity. This makes estimating supremum norm between marginal densities (as in \cite{hao_strong_2024,wang2023singular}) (of two weak solutions) not applicable in our case. However, using Wasserstein metric to estimate the difference between marginal distributions (of two weak solutions) is also not applicable due to the presence of pointwise density $\ell_t (X_t)$ in $b$. To overcome these difficulties, we will estimate weighted total variation distance between marginal densities.
\end{enumerate}

We recall notions of a solution:

\begin{definition}
\begin{enumerate}
\item A \textit{strong solution} to \eqref{main_eq1} is a continuous $\bR^d$-valued process $(X_t)$ on $(\Omega, \cA, \bP)$ such that for $t \in \bT$: $X_t$ is $\cF_t$-adapted, $X_t$ has a distribution $\mu_t \in \sP_p (\bR^d)$, $X_t$ admits a density $\ell_t$, and
\begin{equation} \label{sol_def}
\begin{split}
X_t = X_0 + \int_0^t b(s, X_s, \ell_s (X_s), \mu_s) \diff s & + \int_0^t \sigma (s, X_s, \mu_s) \diff B_s
\qtext{$\bP$-a.s.,} \\
\int_0^t \bE [ |b(s, X_s, \ell_s (X_s), \mu_s)| & + |\sigma (s, X_s, \mu_s)|^2 ] \diff s < \infty.
\end{split}
\end{equation}

\item A \textit{weak solution} to \eqref{main_eq1} is a continuous $\bR^d$-valued process $(X_t)$ on some PS $(\Omega, \cA, \bP)$ on which there exist some $m$-BM $(B_t)$ and some AF $(\cF_t)$ such that the conditions in (1) are satisfied.

\item SDE \eqref{main_eq1} has \textit{strong uniqueness} if, whenever the PS, the AF and the $m$-BM are fixed, two strong solutions $(X_t)$ and $(X'_t)$ such that $X_0 = X'_0$ coincide $\bP$-a.s. on the path space $C(\bT; \bR^d)$. SDE \eqref{main_eq1} has \textit{weak uniqueness} if two weak solutions with the same initial distribution induce the same distribution on $C(\bT; \bR^d)$.

\item SDE \eqref{main_eq1} is \textit{strongly well-posed} if it has strong solution and strong uniqueness. SDE \eqref{main_eq1} is \textit{weakly well-posed} if it has weak solution and weak uniqueness. SDE \eqref{main_eq1} is \textit{well-posed} if it is both strongly and weakly well-posed.
\end{enumerate}
\end{definition}

In \zcref{main_result}, we state our main results about existence and uniqueness of a solution to \eqref{main_eq1}. In \zcref{preliminary}, we remind some facts about optimal transport. Also, we recall estimates of marginal density and establish those of marginal distribution for classical SDEs. We prove our theorems in \zcref{main_thm1:proof} and \zcref{main_thm2:proof} respectively.

Throughout this paper, we use the following conventions:
\begin{enumerate}
\item The set $\bR^m \otimes \bR^n$ is the space of matrices of size $m \times n$ with real entries. For $x \in \bR^m \otimes \bR^n$ and $y \in \bR^n \otimes \bR^k$, let $xy$ be their matrix product. For $x, y \in \bR^m \otimes \bR^n$, let $\langle x, y \rangle$ be their Frobenius inner product and $|x|$ the induced Frobenius norm of $x$.
\item Let $\bR_+ \coloneq \{x \in \bR : x \ge 0\}$. We denote by $\cB (\bR^d)$ the Borel $\sigma$-algebra on $\bR^d$. For brevity, we write $\infty$ for $+\infty$. We denote $x \vee y \coloneq \max \{x, y\}$ and $x \wedge y \coloneq \min \{x, y\}$ for $x, y \in \bR$.
\item We denote by $\nabla, \nabla^2$ the gradient and the Hessian with respect to (w.r.t) the spatial variable. We denote by $\partial_t$ the derivative w.r.t time.
\item Let $L^p (\bR^d)$ be the Lebesgue space of real-valued $p$-integrable functions on $\bR^d$.
\item Let $\sP (\bR^d)$ be the space of Borel probability measures on $\bR^d$. The weak topology (and thus weak convergence $\rightharpoonup$) of $\sP (\bR^d)$ is the topology induced by $C_b (\bR^d)$. The weak$^*$ topology (and thus weak$^*$ convergence $\overset{\ast}{\rightharpoonup}$) of $\sP (\bR^d)$ is the topology induced by $C_c (\bR^d)$.
\end{enumerate}

\section{Main results} \label{main_result}

We recall a class of functions locally integrable in space-time. Let $p, q \in [1, \infty]$. The localized version $\tilde L^p (\bR^d)$ of $L^p (\bR^d)$ is defined by the norm
\[
\|f\|_{\tilde L^p} \coloneq \sup_{x \in \bR^d} \| 1_{B (x, 1)} f\|_{L^p},
\]

Above, $B (x, r)$ is the open ball centered at $x$ with radius $r$. For $0\le t_0 < t_1 \le T$, we define the Bochner space
\[
L^p_q (t_0, t_1) \coloneq L^q ([t_0, t_1]; L^p ({\bR}^d)) .
\]

The localized version $\tilde L^p_q (t_0, t_1)$ of $L^p_q (t_0, t_1)$ is defined by the norm
\[
\| g \|_{\tilde L^p_q (t_0, t_1)} \coloneq \sup_{x \in \bR^d} \| 1_{B(x, 1)} g \|_{L^p_q (t_0, t_1)}.
\]

Then $\|g \|_{\tilde L^\infty_\infty (t_0, t_1)} = \|g \|_{L^\infty_\infty (t_0, t_1)}$. It holds for $p, q \in [1, \infty)$ that
\begin{align}
\| g \|_{L^p_q (t_0, t_1)} &= \bigg ( \int_{t_0}^{t_1} \left ( \int_{\bR^d} | g(s, y) |^p \diff y \right )^{\frac{q}{p}} \diff s \bigg )^{\frac{1}{q}}, \\ 
\| g \|_{\tilde L^p_q (t_0, t_1)} &= \sup_{x \in \bR^d} \bigg ( \int_{t_0}^{t_1} \bigg ( \int_{B(x, 1)} | g(s, y) |^p \diff y \bigg )^{\frac{q}{p}} \diff s \bigg )^{\frac{1}{q}}.
\end{align}

For brevity, we denote
\[
L^p_q (t) \coloneq L^p_q (0, t), \quad \tilde L^p_q (t) \coloneq \tilde L^p_q (0, t), \quad L^p_q \coloneq L^p_q (0, T), \quad \tilde L^p_q \coloneq \tilde L^p_q (0, T).
\]

The class $\sK$ of exponent parameter is defined by
\[
\sK\coloneq \left \{ (p, q) \in (2, \infty]^2 : \frac{d}{p} + \frac{2}{q} < 1 \right \}.
\]

We denote by $M_p (\varrho)$ the $p$-th moment of $\varrho \in \sP (\bR^d)$, i.e., $M_p (\varrho) \coloneq \int_{\bR^d} |x|^p \diff \varrho (x)$.

\subsection{Main results}

Below, we introduce the main assumption about the initial distribution and the coefficients of \eqref{main_eq1}.  Let $a \coloneq \sigma \sigma^\top$. We denote $b_t (x, r , \varrho) \coloneq b(t, x, r , \varrho), \sigma_t(x, \varrho) \coloneq \sigma (t, x, \varrho)$ and $a_t(x, \varrho) \coloneq a (t, x, \varrho)$.

\begin{assumption} \label{main_assmpt1}
There exist constants $p \in [1, \infty), \beta \in (0, 1), C > 0$ and, for $i \in \{0, 1, \ldots, l\}$, function $1 \le f_i \in \tilde L^{p_i}_{q_i}$ with $(p_i, q_i) \in \sK$ such that for all $t \in \bT; x, y \in \bR^d; r, r' \in \bR_+$ and $\varrho, \varrho' \in {\sP}_p ({\bR}^d)$:
\begin{enumerate}
\item $a_t (x, \varrho)$ is invertible.
\item $|b_t (x, r, \varrho)| \le f_0 (t, x)$.
\item $x \mapsto \sigma_t (x, \varrho)$ is weakly differentiable and $| \nabla \sigma_t (x, \varrho)| \le \sum_{i=1}^l f_i (t, x)$.
\item $\nu \in {\sP}_{p} ({\bR}^d)$ has a density $\ell_\nu \in L^\infty (\bR^d)$.
\item
\begin{align}
| \sigma_t (x, \varrho) | + | a_t^{-1} (x, \varrho) | & \le C, \\
| b_t (x, r, \varrho) - b_t (x, r', \varrho) | & \le C \{ |r - r'| + W_p (\varrho, \varrho') \}, \\
| \sigma_t (x, \varrho) - \sigma_t (y, \varrho') | & \le C \{ |x-y|^\beta + W_p (\varrho, \varrho') \} .
\end{align}
\end{enumerate}
\end{assumption}

We gather parameters about $(b, \sigma)$ in \zcref{main_assmpt1}:
\[
\Theta_1 \coloneq (p, d, T, \beta, C, l, (p_i, q_i, f_i)_{i=0}^l).
\]

There is no continuity condition w.r.t spatial variable of $b$. \zcref{main_assmpt1}(2) means that marginal density and marginal distribution do not affect local integrability of the drift. If $b$ is bounded, then it satisfies \zcref{main_assmpt1}(2).

Our main results are the following:

\begin{theorem}[Existence] \label{main_thm1}
Let \zcref{main_assmpt1} hold. Then the following two statements hold:
\begin{enumerate}
\item The SDE \eqref{main_eq1} has a strong solution $(X_t)$ whose marginal distribution is denoted by $(\mu_t)$ and marginal density is denoted by $(\ell_t)$.
\item There exist constants $c_1>0$ (depending on $\Theta_1$), $c_2 >0$ (depending on $\Theta_1, \nu$), and $\delta \in (0, \frac{1}{2})$ (depending on $q_0$) such that
\begin{align}
\sup_{t \in \bT} \|\ell_t\|_{\infty} &\le c_1 \|\ell_\nu\|_{\infty}, \\
W_p (\mu_s, \mu_t) &\le c_2 |t-s|^{\delta} \qtextq{for} s, t \in \bT.
\end{align}
\end{enumerate}
\end{theorem}

\begin{theorem}[Uniqueness] \label{main_thm2}
Let \zcref{main_assmpt1} hold. Assume in addition that $p=1, \sigma_t (x, \varrho) = \sigma_t (x)$ and $|b_t (x, r, \varrho) | \le C$ for all $(t, x, r, \varrho) \in \bT \times \bR^d \times \bR_+ \times \sP_p (\bR^d)$. Then the following two statements hold:
\begin{enumerate} 
\item For $k \in \{1, 2\}$, let $(X^k_t, t \in \bT)$ be a weak solution to \eqref{main_eq1}, $\nu_k$ its initial distribution satisfying \zcref{main_assmpt1}(4) and $(\ell^k_t, t \in \bT)$ its marginal density. Then there exists an increasing function $\Lambda : \bR_+ \to \bR_+$ (depending on $\Theta_1$) such that
\begin{align}
& \sup_{t \in \bT} \int_{\bR^d} ( 1+|x|) | \ell^1_t (x) - \ell^2_t (x) | \diff x \\
\le  & \Lambda (\| \ell_{\nu_1} \|_{\infty} + M_1 (\nu_1)) \int_{\bR^d} (1+|x|) | \ell_{\nu_1} (x) - \ell_{\nu_2} (x) | \diff x.
\end{align}
\item The SDE \eqref{main_eq1} has both weak and strong uniqueness.
\end{enumerate}
\end{theorem}

The condition $p=1$ in \zcref{main_thm2} is crucial for an application of Gronwall's lemma in its proof.

\subsection{Outline of the proofs}

We will summarize the main ideas behind our mollifying argument. Let $(\rho^n)$ be a sequence of mollifiers. We consider the SDE
\begin{equation} \label{outline:eq1}
\begin{cases} 
\diff X^n_t = b(t, X^n_t, \{ \rho^n * \mu_t^n \} ( X_t^{n}), \mu^n_t) \diff t + \sigma (t, X^n_t, \mu^n_t) \diff B_t , \\
\text{$\nu$ is the distribution of $X^n_0$, and $\mu^n_t$ is that of $X^n_t$} .
\end{cases} 
\end{equation}

Above, $*$ is the convolution operator, i.e.,
\[
\{ \rho^n * \varrho \} (x) \coloneq \int_{\bR^d} \rho^n (x-y) \diff \varrho (y)
\qtextq{for every} \varrho \in \sP (\bR^d) .
\]

Then \eqref{outline:eq1} is well-posed and each $X^n_t$ admits a density $\ell^n_t$. The sequence $(\ell^n)$ is locally Hölder continuous on $(0, T] \times \bR^d$. By Arzelà–Ascoli theorem, we can extract a subsequence that converges to some function $\ell : \bT \times \bR^d \to \bR_+$ uniformly on every strip $[ R^{-1}, T] \times B(0, R)$ where $R>0$. We then verify that $\ell_t \coloneq \ell (t, \cdot)$ is indeed a density whose induced distribution $\mu_t \in \sP_p (\bR^d)$. Also, $\mu^n_t$ converges to $\mu_t$ in $W_p$ (as $n \to \infty$) uniformly for $t \in [ R^{-1}, T]$ where $R>0$. By Itô's lemma, $(\mu^n_t, t \in \bT)$ satisfies the Fokker-Planck equation
\begin{equation} \label{outline:eq2}
\partial_t \ell^n_t (x) = - \partial_{x_i} \{ b(t, x, \{ \rho^n * \ell_t^n \} (x), \mu^n_t) \ell^n_t (x) \} + \frac{1}{2} \partial_{x_i} \partial_{x_j} \{ a^{i, j} (t, x, \mu^n_t) \ell^n_t (x) \} .
\end{equation}

Above, we adapt Einstein summation convention. By the convergence of $\ell^n_t$ to $\ell_t$, that of $\mu^n_t$ to $\mu_t$, and the continuity of $b(t, x, r, \varrho)$ in $(r, \varrho)$, we deduce that $(\mu_t)$ satisfies
\begin{equation} \label{outline:eq3}
\partial_t \ell_t (x) = - \partial_{x_i} \{ b(t, x, \ell_t (x), \mu_t) \ell (t, x) \} + \frac{1}{2} \partial_{x_i} \partial_{x_j} \{ a^{i, j} (t, x, \mu_t) \ell_t (x) \}.
\end{equation}

Notice that \eqref{outline:eq3} is the Fokker-Planck equation associated with \eqref{main_eq1}. By superposition principle (e.g. \cite[Section 2]{barbu2020nonlinear}), \eqref{main_eq1} has a weak solution.

\section{Preliminaries} \label{preliminary}

\subsection{Some facts from optimal transport}

Let $\mu, \nu \in \sP (\bR^d)$. The set of \textit{transport plans} (or \textit{couplings}) between $\mu$ and $\nu$ is defined by
\[
\Gamma (\mu, \nu) \coloneq \{ \varrho \in \sP ({\bR}^d \times {\bR}^d ) : \mu = \pi^1_\sharp \varrho \text{ and } \nu = \pi^2_\sharp \varrho \} .
\]

Above, $\pi^i$ is the projection of $\bR^d \times \bR^d$ onto its $i$-th coordinate, and $\pi^i_\sharp \varrho \in \sP (\bR^d)$ is the push-forward measure of $\varrho$ through $\pi^i$ , i.e., $\{ \pi^i_\sharp \varrho \} (A) = \varrho (\{ \pi^i \}^{-1} (A))$ for all $A \in \cB (\bR^d)$. For $\mu, \nu \in \sP_p (\bR^d)$, we define
\[
W_p (\mu, \nu) \coloneq \inf_{\varrho \in \Gamma (\mu, \nu)} \left \{ \int_{\bR^d} |x-y|^p \diff \varrho (x, y) \right \}^{1/p}.
\]

By \cite[Theorem 6.18]{villani2009optimal}, $(\sP_p (\bR^d), W_p)$ is a Polish space. By \cite[Theorem 6.9]{villani2009optimal}, it holds for $\mu_n, \mu \in \sP_p (\bR^d)$ that $W_p (\mu_n, \mu) \to 0$ i.f.f $\mu_n \rightharpoonup \mu$ and $M_p (\mu_n) \to M_p (\mu)$. Let $\Phi_p$ be the collection of all $(\varphi, \psi) \in C_b (\bR^d) \times C_b (\bR^d)$ such that $\varphi (x) + \psi (y) \le |x-y|^p$ for $x, y \in \bR^d$. For brevity, we denote $W_p^p (\mu, \nu) \coloneq (W_p (\mu, \nu))^p$. We denote by $|\mu-\nu|$ the variation of the signed measure $\mu-\nu$ as in \cite[Section 6.1]{rudin_real_1987}. By \cite[Theorems 6.2 and 6.4]{rudin_real_1987}, $|\mu-\nu|$ is a non-negative finite measure. We recall properties needed to prove our theorems:
\begin{lemma} \label{optimal_transport}
\begin{enumerate}
\item \cite[Theorem 1.14]{villani2003topics} It holds for $\mu, \nu \in \sP_1 (\bR^d)$ that
\[
W_1 (\mu, \nu) = \sup_{\substack{ f \in L^1 (|\mu-\nu|) \\ [f]_1 \le 1}} \int_{\bR^d} f (x) \diff \{ \mu - \nu \} (x) ,
\]
where
\[
[f]_1 \coloneq \sup_{\substack{ x,y \in \bR^d \\ x \neq y}} \frac{|f(x)-f(y)|}{|x-y|} .
\]
\item \cite[Remark 7.1.2]{villani2003topics} It holds for $1\le p \le q < \infty$ and $\mu, \nu \in \sP_q (\bR^d)$ that $W_p (\mu, \nu) \le W_q (\mu, \nu)$.
\end{enumerate}
\end{lemma}

Above, the first claim is called Kantorovich duality while the second one is called Kantorovich-Rubinstein theorem. For more information about optimal transport, we refer to \cite{figalli_invitation_2021,maggi2023optimal,villani2009optimal,villani2003topics,Ambrosio2013,santambrogio2015optimal,ambrosio2021lectures,thorpe2018introduction}. The next result states that $W^p_p$ is controlled by weighted total variation distance.

\begin{lemma} \label{wp_lp}
Let $\mu, \nu \in \sP_p (\bR^d)$ be absolutely continuous with corresponding densities $\ell_\mu, \ell_\nu$. Then
\[
W_p^p (\mu, \nu) \le (1 \vee 2^{p-1}) \int_{\bR^d} |x|^p \times |\ell_\mu-\ell_\nu| (x) \diff x.
\]
\end{lemma}

The proof of \zcref{wp_lp} is straightforward. For the sake of completeness, we include its proof in \zcref{appendix}.

\subsection{Moment estimates of marginal distribution}

In the rest of \zcref{preliminary}, we consider measurable functions
\begin{align}
b &: \bT \times \bR^d \to \bR^d, \\
\sigma &: \bT \times \bR^d \to \bR^d \otimes \bR^m.
\end{align}

Let $a \coloneq \sigma \sigma^\top$. We denote $b_t \coloneq b (t, \cdot), \sigma_t \coloneq \sigma (t, \cdot)$ and $a_t \coloneq a (t, \cdot)$. Let $\nu \in \sP (\bR^d)$. We consider the SDE
\begin{equation} \label{eq2}
\begin{cases}
\diff X_t = b (t, X_t) \diff t + \sigma (t, X_t) \diff B_t , \\
\text{$\nu$ is the distribution of $X_0$.}
\end{cases}
\end{equation}

We consider the following set of assumption:
\begin{assumption} \label{drift_noise3}
The following conditions hold:
\begin{enumerate}
\item $a_t (x)$ is invertible for every $(t, x) \in \bT \times \bR^d$.
\item There exist measurable maps $b^{(0)} : \bT \times \bR^d \to \bR^d$ and $b^{(1)} : \bT \times \bR^d \to \bR^d$ such that $b_t (x) = b_t^{(0)} (x) + b_t^{(1)} (x)$ for every $(t, x) \in \bT \times \bR^d$.
\item There exists $1 \le f_0 \in \tilde L^{p_0}_{q_0}$ with $(p_0, q_0) \in \sK$ such that $|b^{(0)}_t (x)| \le f_0 (t, x)$ for every $(t, x) \in \bT \times \bR^d$.
\item There exist constants $\beta \in (0, 1), C > 0$ such that for all $t \in \bT$ and $x, y \in \bR^d$:
\begin{align}
|b_t^{(1)} (x) - b_t^{(1)} (y)| &\le C |x-y|, \\
|b_t^{(1)} (0)| + | \sigma_t (x) | + | a_t^{-1} (x) | & \le C, \\
|\sigma_t(x) - \sigma_t (y)| &\le C |x-y|^{\beta}.
\end{align}
\end{enumerate}
\end{assumption}

Under \zcref{drift_noise3}, \eqref{eq2} has a unique weak solution  (see e.g. \cite[Theorem 2.1(1)]{wang2023singular}). \zcref{drift_noise3} is appealing because it is a general but sufficient condition to obtain Krylov's and Khasminskii's estimates \cite{krylov2008controlled}, which are a key ingredient for establishing the other estimates in the remaining of \zcref{preliminary}. We gather parameters in \zcref{drift_noise3}:
\[
\Theta_3 \coloneq (d, T, \beta, C, p_0, q_0, f_0).
\]

We define the class $\bar \sK$ of exponent parameter by
\[
\bar \sK \coloneq \left \{ (p, q) \in (1, \infty)^2 : \frac{d}{p} + \frac{2}{q} < 2 \right \}.
\]

\begin{remark} \label{exponent_class}
If $f \in \tilde L^p_q$ for some $(p, q) \in \sK$ then there exists $(\bar p, \bar q) \in \bar \sK$ such that $|f|^2 \in \tilde L^{\bar p}_{\bar q}$.
\end{remark}

First, we recall the following estimates:

\begin{lemma}[see e.g. {\cite[Lemma 4.1]{xia2020lq}}] \label{krylov_khasminskii_estimate}
Let $(b, \sigma)$ satisfy \zcref{drift_noise3} and $(X_t)$ be the unique weak solution to \eqref{eq2}. We fix $(p, q) \in \bar K$. Then the following two statements hold:
\begin{enumerate}
\item \textup{(Khasminskii's estimate)} There exist constants $c>0, k >1$ (depending on $\Theta_3, p, q$) such that for $0 \le t_0 < t_1 \le T$ and $g \in \tilde L^{p}_{q}(t_0, t_1)$:
\begin{equation} \label{krylov_khasminskii_estimate:ineq0a}
\bE \bigg [ \exp \left ( \int_{t_0}^{t_1} |g(s, X_s)| \diff s \right ) \bigg | \mathcal F_{t_0} \bigg ] \le \exp(c(1+ \|g\|_{\tilde L^{p}_{q} (t_0, t_1)}^k)).
\end{equation}

\item \textup{(Krylov's estimate)} For $j \in [1, \infty)$, there exists a constant $c>0$ (depending on $\Theta_3, p, q, j$) such that for $0 \le t_0 < t_1 \le T$ and $g \in \tilde L^{p}_{q}(t_0, t_1)$:
\begin{equation} \label{krylov_khasminskii_estimate:ineq0b}
\bE \bigg [ \left ( \int_{t_0}^{t_1} |g(s, X_s)| \diff s \right )^j \bigg | \mathcal F_{t_0} \bigg ] \le c \|g\|_{\tilde L^{p}_{q} (t_0, t_1)}^{j}.
\end{equation}
\end{enumerate}
\end{lemma}

For the sake of completeness, we include its proof in \zcref{appendix}. Second, we establish the following moment estimates:

\begin{theorem} \label{moment_estimate}
Let $(b, \sigma)$ satisfy \zcref{drift_noise3} and $(X_t)$ be the unique weak solution to \eqref{eq2}. Then the following two statements hold:
\begin{enumerate}
\item $X_t$ admits a density for $t \in (0, T]$.

\item Let $p \in [1, \infty)$. There exist constants $c > 0$ (depending on $\Theta_3, p$) and $\delta \in (0, \frac{1}{2})$ (depending on $q_0$) such that for $0 \le u \le t \le T$:
\begin{align}
\bE \big [ \sup_{s \in [u, t]} |X_s|^p \big ] &\le c (1+ \bE [|X_u|^p]), \label{moment_estimate:ineq1} \\
\bE \big [ \sup_{s \in [u, t]} |X_s - X_u|^p \big ] &\le c |t-u|^{\delta p} (1+\bE [|X_u|^p]) . \label{moment_estimate:ineq2} 
\end{align}
\end{enumerate}
\end{theorem}

\begin{proof}
By \zcref{drift_noise3}(2), there exist measurable maps $b^{(0)} : \bT \times {\bR}^d \to {\bR}^d$ and $b^{(1)} : \bT \times {\bR}^d \to {\bR}^d$ such that $b_t (x) = b_t^{(0)} (x) + b_t^{(1)} (x)$. By \zcref{drift_noise3}(3), there exists $f_0 \in \tilde L^{p_0}_{q_0}$ with $(p_0, q_0) \in \sK$ such that $|b^{(0)}_t (x)| \le f_0 (t, x)$. We consider the SDE
\begin{equation} \label{moment_estimate:eq1}
\diff \bar X_t = b^{ (1)} (t, \bar X_t) \diff t + \sigma (t, \bar X_t) \diff B_t .
\end{equation}

Above, the distribution of $\bar X_0$ is $\nu$. Clearly, $(b^{ (1)}, \sigma)$ satisfies \zcref{drift_noise3}, so \eqref{moment_estimate:eq1} is well-posed. We define
\begin{align}
\xi_t & \coloneq \{ \sigma^\top_t a_t^{-1} b_t^{ (0)} \} (\bar X_t),\\
\bar B_t & \coloneq B_t - \int_0^t \xi_s \diff s, \\
R_t & \coloneq \exp \bigg ( \int_0^t \xi_s^\top \diff B_s - \frac{1}{2} \int_0^t |\xi_s|^2 \diff s \bigg ), \\
I_t & \coloneq \bE \bigg [ \exp \bigg ( \frac{1}{2} \int_0^t |\xi_s|^2 \diff s \bigg ) \bigg ] .
\end{align}

By uniform boundedness of $\sigma^\top_t a_t^{-1}$, \zcref{exponent_class} and \zcref{krylov_khasminskii_estimate}(1), we have $I_T < \infty$. So $R_T$ is an exponential martingale with $\bE [R_T]=1$. By Girsanov theorem, $(\bar B_t, t \in \bT)$ is an $m$-dimensional Brownian motion under the probability measure $\bar \bP \coloneq R_T \bP$. We denote by $\bar \bE$ the expectation w.r.t $\bar \bP$. Clearly, \eqref{moment_estimate:eq1} can be written under $\bar \bP$ as
\begin{equation} \label{moment_estimate:eq2}
\diff \bar X_t = b (t, \bar X_t) \diff t + \sigma (t, \bar X_t) \diff \bar B_t .
\end{equation}

\begin{enumerate}
\item By \zcref{drift_noise3}(4) and \cite[Theorem 1.2]{menozzi2021density}, the distribution of $\bar X_t$ under $\bP$ admits a density. Notice that $\bar \bP$ and $\bP$ are equivalent, so the distribution of $\bar X_t$ under $\bar \bP$ also admits a density. Because $\bar X_0$ is $\mathcal F_0$-measurable, it holds for $\varphi \in C_c^\infty (\bR^d)$ that
\[
\bar \bE [ \varphi ( \bar X_0 ) ] = \bE [ \varphi ( \bar X_0 ) R_0 ] = \bE [ \varphi ( \bar X_0 )].
\]

Then $\nu$ is also the distribution of $\bar X_0$ under $\bar \bP$. By weak uniqueness of \eqref{eq2} and \eqref{moment_estimate:eq2}, the distribution of $X_t$ under $\bP$ is the same as that of $\bar X_t$ under $\bar \bP$. Thus the distribution of $X_t$ under $\bP$ admits a density.

\item The estimate \eqref{moment_estimate:ineq1} follows from \cite[Theorem 1.1(1)]{huang2022singular}. It remains to prove \eqref{moment_estimate:ineq2}. Recall that $f_0$ is a parameter in \zcref{drift_noise3}. We have
\begin{align}
\bE \big [ \sup_{s \in [u, t]} | X_s - X_u |^p \big ] & \lesssim \begin{myaligned}[t]
& \bE \bigg [ \bigg ( \int_u^t f_0 (r, X_r) \diff r \bigg )^p \bigg ] \\
& + \bE \bigg [ \bigg ( \int_u^t |b^{(1)} (r, X_r)| \diff r \bigg )^p \bigg ] \\
& + \bE \bigg [ \sup_{s \in [u, t]}  \left | \int_u^s \sigma (r, X_r) \diff B_r \right |^p \bigg ]
\end{myaligned} \\
&\eqcolon J_1 + J_2 + J_3.
\end{align}

There exists $\bar q_0 \in (2, q_0)$ such that $(p_0, \bar q_0) \in \bar K$. Let $\delta \coloneq \frac{1}{\bar q_0} - \frac{1}{q_0} \in (0, \frac{1}{2})$. By Hölder's inequality,
\[
\|f_0\|_{\tilde L^{p_0}_{\bar q_0} (u, t)} \le (t - u)^{\delta} \|f_0\|_{\tilde L^{p_0}_{q_0} (u, t)}.
\]

By \zcref{krylov_khasminskii_estimate}(2),
\[
J_1 \lesssim \|f_0\|_{\tilde L^{p_0}_{\bar q_0} (u, t)}^p \le (t - u)^{\delta p} \|f_0\|_{\tilde L^{p_0}_{q_0} (u, t)}^p \lesssim (t - u)^{\delta p}.
\]

As for $I_2$ and $I_3$, we have
\begin{align}
J_2 & \lesssim  |t-u|^{p-1} \bE \left [ \int_u^t (1+|X_r|^p) \diff r \right ] \\
& \lesssim |t-u|^{p} (1 + \bE [ |X_u|^p ]) \qtext{by \eqref{moment_estimate:ineq1}}, \\
J_3 & \lesssim |t-u|^{\frac{p}{2}}.
\end{align}

Then
\[
\bE \big [ \sup_{s \in [u, t]} | X_s - X_u |^p \big ] \lesssim |t-u|^{\delta p} (1 + \bE [ |X_u|^p ]).
\]

The estimate \eqref{moment_estimate:ineq2} then follows. This completes the proof.
\end{enumerate}
\end{proof}

\subsection{Heat kernel estimates}

For $(s, x) \in [0, T) \times \bR^d$, we consider the SDE
\begin{equation} \label{eq1}
\begin{cases} 
\diff X^x_{s,t} &= b (t, X^x_{s, t}) \diff t + \sigma (t, X^x_{s, t}) \diff B_t \qtextq{for} t \in (s, T], \\
X^x_{s,s} &= x .
\end{cases}
\end{equation}

If $X^x_{s,t}$ admits a density, we denote it by $p^{b, \sigma}_{s, t}(x , \cdot)$. We consider the following set of assumption:
\begin{assumption} \label{drift_noise4}
The following conditions hold:
\begin{enumerate}
\item $a_t (x)$ is invertible for every $(t, x) \in \bT \times \bR^d$.
\item There exists $1 \le f_0 \in \tilde L^{p_0}_{q_0}$ with $(p_0, q_0) \in \sK$ such that $|b_t (x)| \le f_0 (t, x)$ for every $(t, x) \in \bT \times \bR^d$.
\item There exist constants $\beta \in (0, 1), C > 0$ such that for all $t \in \bT$ and $x, y \in \bR^d$:
\begin{align}
| \sigma_t (x) | + | a_t^{-1} (x) | & \le C, \\
|\sigma_t(x) - \sigma_t (y)| &\le C |x-y|^{\beta}.
\end{align}
\end{enumerate}
\end{assumption}

Notice that \zcref{drift_noise4} is \zcref{drift_noise3} with $b^{(1)} = 0$. As such, under \zcref{drift_noise4}, \eqref{eq1} has a unique weak solution; and for $t \in (s, T]$, $X^x_{s,t}$ admits a density (by \zcref{moment_estimate}(1)). In this case, $\{ p^{b, \sigma}_{s, t} : 0 \le s < t \le T \}$ are transition densities associated with \eqref{eq1}. We gather parameters from \zcref{drift_noise4}:
\[
\Theta_4 \coloneq (d, T, \beta, C, p_0, q_0, f_0).
\]

For $\lambda>0$ and $\gamma \in \bR$, the heat kernel $p^{\gamma , \lambda}$ is defined for $t>0$ and $x \in \bR^d$ by
\begin{equation} \label{gaussian_heat_kernel}
p^{\gamma , \lambda}_t (x) \coloneq \frac{1}{t^{(\gamma + d) / 2}} e^{-\frac{\lambda|x|^2}{t}} .
\end{equation}

First, we recall the following estimates:

\begin{lemma} \label{Holder_estimate}
\cite[Lemma 3.9]{zhao_existence_2024}
Let $(b, \sigma)$ satisfy \zcref{drift_noise4}. Let $\gamma_0 := 1-\frac{d}{p_0}-\frac{2}{q_0}$. Then the map $x \mapsto p^{b, \sigma}_{s, t}(x , y)$ is differentiable for every $(t, y) \in (s, T] \times \bR^d$. Moreover, there exist constants $c_1, c_2, c_3>0$ and $\lambda \in (0, 1)$ depending on $\Theta_4$ such that
\begin{enumerate}
\item (Gaussian estimate) for $0 \le s < t \le T$ and $x, y \in \bR^d$:
\[
c_1 p^{0, \lambda^{-1}}_{t-s} (y-x) \le p^{b, \sigma}_{s, t}(x , y) \le c_2 p^{0, \lambda}_{t-s} (y-x) .
\]

\item (Gradient estimate in $x$) for $0 \le s < t \le T$ and $x, y \in \bR^d$:
\[
|\nabla_x p^{b, \sigma}_{s, t} (x , y) | \le c_3 p^{1, \lambda}_{t-s} (y-x) .
\]

\item (Hölder estimate in $t$ and $y$) for $\gamma \in (0, \beta \wedge \gamma_0)$, there exists a constant $c_4 >0$ depending on $(\Theta_4, \gamma)$ such that for $0 \le s < t_1 < t_2 \le T$ and $x, y, y_1, y_2 \in \bR^d$:
\begin{align}
|p^{b, \sigma}_{s, t_2} (x , y) - p^{b, \sigma}_{s, t_1} (x , y)| & \le c_4 |t_1-t_2|^{\frac{\gamma}{2}} \sum_{i=1}^2 p^{\gamma, \lambda}_{t_i-s} (x-y) , \\
| p^{b, \sigma}_{s, t} (x , y_1) - p^{b, \sigma}_{s, t} (x , y_2) | & \le c_4 |y_1-y_2|^{\gamma} \sum_{i=1}^2 p^{\gamma, \lambda}_{t-s} (x-y_i) .
\end{align}
\end{enumerate}
\end{lemma}

We will need the following results:

\begin{corollary} \label{Holder_estimate2}
Let $(b, \sigma)$ satisfy \zcref{drift_noise4} and $(X_t)$ be the unique weak solution to \eqref{eq2}. Let $\nu$ admit a density $\ell_\nu \in L^\infty (\bR^d)$. We denote by $\ell_t$ the density of $X_t$. Then the following two statements hold:
\begin{enumerate}
\item There exists a constant $c_1 > 0$ depending on $\Theta_4$ such that:
\[
\sup_{t \in \bT} \| \ell_t \|_\infty \le c_1 \| \ell_\nu \|_\infty .
\]

\item Let $\gamma_0 := 1-\frac{d}{p_0}-\frac{2}{q_0}$. For $\gamma \in (0, \beta \wedge \gamma_0)$, there exists a constant $c_2 >0$ depending on $(\Theta_4, \gamma, \nu)$ such that for $0 < s, t \le T$ and $x, y \in \bR^d$:
\begin{align}
|\ell_{t} (x) - \ell_{s} (x)| & \le c_2 |t-s|^{\frac{\gamma}{2}} ( t^{-\frac{\gamma}{2}} + s^{-\frac{\gamma}{2}} ) , \\
| \ell_t (x) - \ell_t (y) | & \le c_2 | x -y |^{\gamma} t^{-\frac{\gamma}{2}} .
\end{align}
\end{enumerate}
\end{corollary}

\begin{lemma} \label{duhamel_form1}
Let $(b, \sigma)$ satisfy \zcref{drift_noise4} and $(X_t)$ be the unique weak solution to \eqref{eq2}. Let $\nu$ admit a density $\ell_\nu$. We denote by $\ell_t$ the density of $X_t$. Then it holds for $t \in \bT$ and $x \in \bR^d$:
\begin{align}
\ell_t (x) &= \begin{multlined}[t]
\int_{\bR^d} p_{0, t}^{0, \sigma} (y, x) \ell_{\nu}(y) \diff y + \int_0^t \int_{{\bR^d}} \ell_s (y) \langle b_s (y) ,  \nabla_y p^{0, \sigma}_{s, t} (y, x) \rangle \diff y \diff s.
\end{multlined}
\end{align}
\end{lemma}

\begin{lemma} \label{tail_moment}
Let $p \in \{ 0 \} \cup [1, \infty)$ and $\nu \in \sP_p (\bR^d)$. Let $(b, \sigma)$ satisfy \zcref{drift_noise4} and $(X_t)$ be the unique weak solution to \eqref{eq2}. We denote by $\mu_t$ the distribution of $X_t$. There exists a function $\phi : \bR_+ \to \bR_+$ depending on $(\Theta_4, \nu, p)$ such that $\lim_{R \to \infty} \varphi (R) = 0$ and that
\[
\sup_{t \in \bT} \int_{B^c (0, R)} | x |^p \diff \mu_t (x) \le \phi (R)
\qtextq{for every} R \ge 0.
\]

Above, $B^c (0, R) \coloneq \bR^d \setminus B(0, R)$ where $B (0, R) \coloneq \{ x \in \bR^d : |x| \le \bR \}$.
\end{lemma}

The proof of \zcref{Holder_estimate2} is a straightforward application of \zcref{Holder_estimate} and the fact that $\ell_t (x) = \int_{\bR^d} \ell_\nu (y) p^{b, \sigma}_{0, t} (y, x) \diff y$. The proof of \zcref{duhamel_form1} is contained in that of \cite[Lemma 4.1]{wang2023singular}. The proof of \zcref{tail_moment} is straightforward application of Markov's inequality and \zcref{Holder_estimate}(1). For the sake of completeness, we still include their proofs in \zcref{appendix}.

\section{Proof of \zcref{main_thm1}} \label{main_thm1:proof}

We recall that $\Theta_1 = (p, d, T, \beta, C, l, (p_i, q_i, f_i)_{i=0}^l)$ contains parameters about $(b, \sigma)$ from \zcref{main_assmpt1}. We write $M_1 \lesssim M_2$ if there exists a constant $c>0$ (depending on $\Theta_1, \nu$) such that $M_1 \le c M_2$. We construct a sequence $(\rho^n)$ of mollifiers as follows. We fix a smooth density $\rho : \bR^d \to \bR$ whose support is contained in $B (0, 1)$. For each $n \in \bN$, we define $\rho^n :\bR^d \to \bR$ by $\rho^n (x) \coloneq n^d \rho (nx)$ and consider the McKean-Vlasov SDE
\begin{equation} \label{main_thm1:eq1}
\begin{cases} 
\diff X^n_t = b(t, X^n_t, (\rho^n * \mu_t^n) ( X_t^{n}), \mu^n_t) \diff t + \sigma (t, X^n_t, \mu^n_t) \diff B_t , \\
\text{$\nu$ is the distribution of $X^n_0$, and $\mu^n_t$ is that of $X^n_t$} .
\end{cases} 
\end{equation}

Then \eqref{main_thm1:eq1} is a mollified version of \eqref{main_eq1}.

\subsection{Stability estimates for mollified SDEs} \label{stability_estimate}

We define the map $b^n : \bT \times \bR^d \times \sP_p (\bR^d) \to \bR^d$ by $b^n (t, x, \varrho) \coloneq b(t, x, (\rho^n * \varrho) (x), \varrho)$. Then
\begin{align}
|b^n (t, x, \varrho) - b^n (t, x, \tilde \varrho)| &\lesssim \left | \int_{\bR^d} \rho^n (x-y) \diff (\varrho - \tilde \varrho) (y) \right | + W_p (\varrho, \tilde \varrho) \label{main_thm1:ineq1} \\
&\le \|\nabla \rho^n\|_\infty W_1 (\varrho, \tilde \varrho) + W_p (\varrho, \tilde \varrho) \label{main_thm1:ineq2} \\
&\le (1 + \|\nabla \rho^n\|_\infty) W_p (\varrho, \tilde \varrho). \label{main_thm1:ineq3} 
\end{align}

Above, \eqref{main_thm1:ineq1} is due to \zcref{main_assmpt1}(5), \eqref{main_thm1:ineq2} due to \zcref{optimal_transport}(1), and \eqref{main_thm1:ineq3} due to \zcref{optimal_transport}(2). It follows that $b^n$ is Lipschitz in distribution variable. We consider the McKean-Vlasov SDE
\begin{equation} \label{main_thm1:eq2}
\begin{cases} 
\diff Y_t = b^n (t, Y_t, \xi_t) \diff t + \sigma (t, Y_t, \xi_t) \diff B_t , \\
\text{$\nu$ is the distribution of $Y_0$, and $\xi_t$ is that of $Y_t$} .
\end{cases} 
\end{equation}

It follows from \cite[Theorem 1.1(1)]{huang2022singular} that \eqref{main_thm1:eq2} is well-posed.

\begin{remark}
The Lipschitz continuity of $b^n (t, x, \cdot)$ is just for applying the existence part of \cite[Theorem 1.1(1)]{huang2022singular}. The size of its Lipschitz constant of $b^n (t, x, \cdot)$ does not play any role below.
\end{remark}

By \eqref{main_thm1:eq1}, $(X^n_t, t \in \bT)$ satisfies \eqref{main_thm1:eq2}. As a consequence, \eqref{main_thm1:eq1} is well-posed. We define the maps $\bar b^n : \bT \times \bR^d \to \bR^d$ and $\bar \sigma^n : \bT \times \bR^d \to \bR^d \otimes \bR^m$ by
\begin{align}
\bar b^n (t, x) & \coloneq b^n (t, x, \mu^n_t), \\
\bar \sigma^n (t, x) & \coloneq \sigma (t, x, \mu^n_t).
\end{align}

We have
\begin{equation} \label{main_thm1:eq4}
\diff X_t^{n} = \bar b^n (t, X_t^{n}) \diff t + \bar \sigma^n (t, X_t^{n}) \diff B_t.
\end{equation}

Let $\bar a^n \coloneq \bar \sigma^n (\bar \sigma^n)^\top$. We denote $\bar b^n_t \coloneq \bar b^n (t, \cdot), \bar \sigma^n_t \coloneq \bar \sigma^n (t, \cdot)$ and $\bar a^n_t \coloneq \bar a^n (t, \cdot)$.

\begin{lemma} \label{same_parameters}
All pairs $(\bar b^n, \bar \sigma^n)_{n \in \bN}$ satisfy \zcref{drift_noise4} for the same set of parameters.
\end{lemma}

\begin{proof}
By \zcref{main_assmpt1}, we have for any $(t, r, \varrho) \in \bT \times \bR_+ \times \sP_p (\bR^d)$ and $x, y \in \bR^d$ that
\begin{align}
| b (t, x, r, \varrho) | & \le C f_0 (t, x) , \\
\|\sigma_t\|_\infty + \|a_t^{-1}\|_\infty & \le C, \\
| \sigma_t (x, \varrho) - \sigma_t (y, \varrho) | & \le C |x-y|^\beta .
\end{align}

Then
\begin{align}
|\bar b^n_t (x)| & = | b^n (t, x, \mu^n_t) | \\
& = | b(t, x, (\rho^n * \mu^n_t) (x), \mu^n_t) | \\
& \le C f_0 (t, x) , \\
| \bar \sigma^n_t (x) - \bar \sigma^n_t (y) | & = | \sigma_t (x, \mu^n_t) - \sigma_t (y, \mu^n_t) | \\
& \le C |x-y|^\beta .
\end{align}

The uniform elipticity of $\bar \sigma^n$ is inherited from that of $\sigma$. This completes the proof.
\end{proof}

By \zcref{moment_estimate}(1), each $X^n_t$ admits a density denoted by $\ell^n_t$. By \zcref{moment_estimate}(2), there exists a constant $\delta \in (0, \frac{1}{2})$ depending on $q_0$ such that
\begin{align}
\sup_{n \in {\bN}} \sup_{t \in \bT} M_p (\mu^n_t) &\lesssim 1, \label{main_thm1:ineq6} \\
\sup_{n \in {\bN}} W_p (\mu^n_s, \mu^n_t) &\lesssim |t-s|^{\delta}
\qtextq{for all}
s, t \in \bT . \label{main_thm1:ineq7} 
\end{align}

Let $\gamma_0 \coloneq 1-\frac{d}{p_0}-\frac{2}{q_0}$ and $\gamma \coloneq \frac{\beta \wedge \gamma_0}{2}$. By \zcref{Holder_estimate2}(1),
\begin{equation} \label{main_thm1:ineq4}
\sup_{n \in {\bN}} \sup_{t \in \bT} \| \ell^n_t \|_{\infty} \lesssim \| \ell_\nu \|_\infty.
\end{equation}

By \zcref{Holder_estimate2}(2), it holds for $t \in (0, T]$ that
\begin{align}
\sup_{n \in {\bN}} \sup_{\substack{s, r \in [t,  T] \\ s \neq r} } \sup_{x \in \bR^d} \frac{|\ell^n_{s} (x) - \ell^n_{r} (x)|}{|s-r|^{\frac{\gamma}{2}}} & \lesssim  t^{-\frac{\gamma}{2}} , \label{main_thm1:ineq5a} \\
\sup_{n \in {\bN}} \sup_{s \in [t, T]} \sup_{\substack{x, y \in \bR^d \\ x \neq y}} \frac{| \ell^n_s (x) - \ell^n_s (y) |}{| x -y |^{\gamma}} & \lesssim  t^{-\frac{\gamma}{2}} . \label{main_thm1:ineq5b}
\end{align}

By \zcref{tail_moment}, there exists a function $\phi : \bR_+ \to \bR_+$ depending on $(\Theta_1, \nu)$ such that $\lim_{R \to \infty} \phi (R) = 0$ and that
\begin{equation} \label{main_thm1:ineq8} 
\sup_{n \in {\bN}} \sup_{t \in \bT} \int_{B^c (0, R)} (1 + | x |^p) \diff \mu^n_t (x) \le \phi (R)
\qtextq{for every} R \ge 0.
\end{equation}

By \eqref{main_thm1:ineq7}, the map $\bT \to \sP_p (\bR^d), \, t \mapsto \mu^n_t$ is continuous.

\subsection{Convergence of marginal densities of mollified SDEs} \label{convergence_marginal_densities}

By \eqref{main_thm1:ineq5a}, \eqref{main_thm1:ineq5b}, Arzelà–Ascoli theorem and diagonal extraction, there exist a sub-sequence (also denoted by $(\ell^n)$ for simplicity) and a continuous function $\ell : \bT \times \bR^d \to \bR_+$ such that
\begin{equation} \label{main_thm1:eq5}
\lim_{n} \sup_{t \in [ R^{-1}, T]} \sup_{x \in B(0, R)} |\ell^n_t (x) - \ell_t (x)| =0
\qtextq{for every}
R > T^{-1}.
\end{equation}

Above, $\ell_t \coloneq \ell (t, \cdot)$. Clearly, $\ell_0 = \ell_\nu$ and
\begin{equation} \label{main_thm1:ineq9}
\sup_{t \in \bT} \| \ell_t \|_{\infty} \lesssim \| \ell_\nu \|_\infty .
\end{equation}

We remark that the constant in \eqref{main_thm1:ineq9} depends on $\Theta_1$. Next we verify that $\ell_t$ is indeed a density for every $t \in (0, T]$. We have
\begin{align}
\int_{B(0, R)} \ell^n_t (x) \diff x &= 1- \int_{B^c(0, R)} \ell^n_t (x) \diff x \\
&\gtrsim 1 - \phi (R) \qtext{by \eqref{main_thm1:ineq8}}.
\end{align}

By \eqref{main_thm1:eq5}, \eqref{main_thm1:ineq9} and DCT,
\[
\int_{B(0, R)} \ell_t (x) \diff x = \lim_{n \to \infty} \int_{B(0, R)} \ell^n_t (x) \diff x.
\]

It follows that
\[
1 - \phi (R) \lesssim \int_{B(0, R)} \ell_t (x) \diff x \le 1 .
\]

Then
\[
\int_{\bR^d} \ell_t (x) \diff x = \lim_{R \to \infty} \int_{B(0, R)} \ell_t (x) \diff x = 1.
\]

Let $\mu_t \in \sP (\bR^d)$ be the probability measure induced by $\ell_t$, i.e.,
\[
\mu_t (B) \coloneq \int_{B} \ell_t (x) \diff x
\qtextq{for every}
B \in \cB (\bR^d). 
\]

\begin{lemma} \label{weak_convergence}
We have for each $t \in \bT$ that $\mu^n_t \rightharpoonup \mu_t$ as $n \to \infty$.
\end{lemma}

\begin{proof}
It suffices to consider $t \in (0, T]$. By \eqref{main_thm1:eq5},
\begin{equation} \label{weak_convergence:ineq0}
\mu^n_t \overset{\ast}{\rightharpoonup} \mu_t \quad \text{as} \quad n \to \infty.
\end{equation}

Let $f \in C_b (\bR^d)$ and $g \in C_c (\bR^d)$ such that $0 \le g \le 1$. Then $gf \in C_c (\bR^d)$ and $f = (1-g) f + gf$. We have
\[
\left | \int_{\bR^d} f (x) \diff \{ \mu^n_t - \mu_t \} (x) \right | \le \begin{myaligned}[t]
& \|f\|_\infty \int_{\bR^d} \{ 1-g (x) \} \diff \{ \mu^n_t + \mu_t \} (x) \\
& + \left | \int_{\bR^d} g (x) f (x) \diff \{ \mu^n_t - \mu_t \} (x) \right |.
\end{myaligned}
\]

By \eqref{weak_convergence:ineq0},
\[
\limsup_n \int_{\bR^d} g (x) f (x) \diff \{ \mu^n_t - \mu_t \} (x) =0 .
\]

Then
\[
\limsup_n \left | \int_{\bR^d} f (x) \diff \{ \mu^n_t - \mu_t \} (x) \right | \le \|f\|_\infty \limsup_n \int_{\bR^d} \{ 1-g (x) \} \diff \{ \mu^n_t + \mu_t \} (x) .
\]

Notice that
\begin{align}
\limsup_n \int_{\bR^d} \{ 1-g (x) \} \diff \mu^n_t (x) &= 1 - \liminf_n \int_{\bR^d} g (x) \diff \mu^n_t (x) \\
&= 1 - \int_{\bR^d} g (x) \diff \mu_t (x) \qtext{by \eqref{weak_convergence:ineq0}} \\
&=  \int_{\bR^d} \{ 1-g (x) \} \diff \mu_t (x) .
\end{align}

Thus
\begin{equation} \label{weak_convergence:ineq1}
\limsup_n \left | \int_{\bR^d} f (x) \diff \{ \mu^n_t - \mu_t \} (x) \right | \le 2 \|f\|_\infty \int_{\bR^d} \{ 1-g (x) \} \diff \mu_t (x) .
\end{equation}

Because $\mu_t$ is a probability measure,
\begin{equation} \label{weak_convergence:eq1}
\sup \left  \{ \int_{\bR^d} g (x) \diff \mu_t (x) : g \in C_c (\bR^d) \text{ and } 0 \le g \le 1 \right \}=1.
\end{equation}

The claim then follows from \eqref{weak_convergence:ineq1} and \eqref{weak_convergence:eq1}.
\end{proof}

By monotone convergence theorem (MCT),
\begin{align}
\int_{B^c (0, R)} |x|^p \diff \mu_t (x) &= \lim_{K \to \infty} \int_{B^c (0, R) \cap B(0, K)} |x|^p \diff \mu_t (x) \\
&= \lim_{K \to \infty} \lim_{n \to \infty} \int_{B^c (0, R) \cap B(0, K)} |x|^p \diff \mu^n_t (x) \qtext{by \eqref{main_thm1:eq5}} \\
&\lesssim \phi (R) \qtext{by \eqref{main_thm1:ineq8}} .
\end{align}

Then
\begin{equation} \label{main_thm1:ineq10a}
\sup_{t \in \bT} \int_{B^c (0, R)} |x|^p \diff \mu_t (x) \lesssim \phi (R).
\end{equation} 

Clearly, $x \mapsto |x|^p$ is continuous and bounded from below. By \zcref{weak_convergence} and Portmanteau's theorem,
\begin{align}
\sup_{t \in \bT} M_p (\mu_t) &\le \sup_{t \in \bT} \liminf_n M_p (\mu^n_t) \\
&\lesssim 1 \quad \text{by \eqref{main_thm1:ineq6}} . \label{main_thm1:ineq11}
\end{align}

Then $\mu_t \in \sP_p (\bR^d)$ for every $t \in \bT$. We have
\begin{align}
W_p (\mu_s, \mu_t) &\le \liminf_n W_p (\mu^n_s, \mu^n_t) \label{main_thm1:eq8} \\
&\lesssim |t-s|^{\delta} \label{main_thm1:ineq12}.
\end{align}

Above, \eqref{main_thm1:eq8} is due to the lower semicontinuity of $W_p$ w.r.t weak convergence; \eqref{main_thm1:ineq12} is due to \eqref{main_thm1:ineq7}. Next we establish an essential result about convergence:

\begin{lemma} \label{wp_convergence}
We have for every $R > T^{-1}$ that
\[
\lim_n \sup_{t \in [ R^{-1}, T]} W_p (\mu^n_t, \mu_t) = 0 .
\]
\end{lemma}

\begin{proof}
We have
\begin{align}
W_p^p (\mu^n_t, \mu_t) &\lesssim \int_{\bR^d} |x|^p \times | \ell^n_t (x) - \ell_t (x) | \diff x \quad \text{by \zcref{wp_lp}} \\
&\le \int_{B(0, k)} |x|^p \times | \ell^n_t (x) - \ell_t (x) | \diff x + \int_{B^c(0, k)} |x|^p \{ \ell^n_t (x) + \ell_t (x) \} \diff x \\
&\eqcolon I (t, n, k) + J (t, n, k)
\qtextq{for every} k > 0.
\end{align}

By \eqref{main_thm1:eq5},
\[
\lim_n \sup_{t \in [ R^{-1}, T]} I (t, n, k) = 0.
\]

By \eqref{main_thm1:ineq8} and \eqref{main_thm1:ineq10a},
\[
\sup_{n \in {\bN}} \sup_{t \in \bT} J (t, n, k) \lesssim \phi(k).
\]

As such,
\begin{align}
\limsup_n \sup_{t \in [ R^{-1}, T]} W_p^p (\mu^n_t, \mu_t) & \lesssim \begin{myaligned}[t]
& \limsup_n \sup_{t \in [ R^{-1}, T]} I (t, n, k) \\
& + \limsup_n \sup_{t \in \bT} J (t, n, k) 
\end{myaligned} \\
& \lesssim  \phi (k) .
\end{align}

The claim then follows by taking the limit $k \to \infty$.
\end{proof}

\subsection{Existence of a weak solution}

Notice that $\rho^n * \mu_t^n = \rho^n * \ell_t^n$. The Fokker-Planck equation (in distributional sense) associated with \eqref{main_thm1:eq1} is
\begin{equation} \label{main_thm1:eq9}
\partial_t \ell^n_t (x) = - \partial_{x_i} \{ b(t, x, \{ \rho^n * \ell_t^n \} (x), \mu^n_t) \ell^n_t (x) \} + \frac{1}{2} \partial_{x_i} \partial_{x_j} \{ a^{i, j} (t, x, \mu^n_t) \ell^n_t (x) \} .
\end{equation}

This means for each $(\varphi, \psi) \in C^\infty_c (0, T) \times C^\infty_c (\bR^d)$ that
\begin{align} 
& - \int_{\bT} \int_{\bR^d} \varphi ' (t) \psi(x) \diff \mu^n_t (x) \diff t \\
= & \begin{myaligned}[t]
& \sum_{i=1}^d \int_{\bT} \int_{\bR^d} b(t, x, \{ \rho^n * \ell_t^n \} (x), \mu^n_t) \varphi (t) \partial_{x_i} \psi (x) \diff \mu^n_t (x) \diff t \\
& + \frac{1}{2} \sum_{i, j = 1}^d \int_{\bT} \int_{\bR^d} a^{i, j} (t, x, \mu^n_t) \varphi (t) \partial_{x_i} \partial_{x_j} \psi (x) \diff \mu^n_t (x) \diff t.
\end{myaligned} \label{main_thm1:eq10}
\end{align}

Above, $a^{i, j}$ is the entry in the $i$-th row and $j$-th column of $a$. We recall from \zcref{wp_convergence} and \eqref{main_thm1:eq5} that
\begin{align}
W_p (\mu^n_t, \mu_t) &\xrightarrow{n \to \infty} 0, \label{main_thm1:eq11} \\
\sup_{x \in B(0, R)} |\ell^n_t (x) - \ell_t (x)| &\xrightarrow{n \to \infty} 0
\qtextq{for every} (t, R) \in \bT \times \bR_+. \label{main_thm1:eq12}
\end{align}

We fix $(\varphi, \psi) \in C^\infty_c (0, T    ) \times C^\infty_c (\bR^d)$. By \eqref{main_thm1:eq11}, the boundedness of $a$, and the continuity of $a$ w.r.t distribution variable,
\begin{equation}
\int_{\bR^d} a^{i, j} (t, x, \mu^n_t) \partial_{x_i} \partial_{x_j} \psi (x) \diff \mu^n_t (x)  \xrightarrow{n \to \infty}  \int_{\bR^d} a^{i, j} (t, x, \mu_t) \partial_{x_i} \partial_{x_j} \psi (x) \diff \mu_t (x).
\end{equation}

Let $S \coloneq B(0, 1) + \supp \psi$. Then $S$ is bounded. By triangle inequality,
\begin{align}
\|1_S \{ [ \rho^n * \ell_t^n ] - \ell_t \} \|_\infty &\le \|1_S \{ \rho^n * [ \ell_t^n- \ell_t ] \} \|_\infty + \|1_S \{ \rho^n *\ell_t- \ell_t \} \|_\infty \\
&\le \| \rho^n * \{ 1_{S} [\ell_t^n- \ell_t] \} \|_\infty + \|1_S \{ \rho^n *\ell_t- \ell_t \} \|_\infty \\
&\le \| 1_{S} [ \ell_t^n- \ell_t ] \|_\infty + \|1_S \{ \rho^n *\ell_t- \ell_t \} \|_\infty.
\end{align}

By \eqref{main_thm1:eq12}, $\| 1_{S} [ \ell_t^n- \ell_t ] \|_\infty \to 0$ as $n \to \infty$. By \cite[Proposition 4.21]{brezis_functional_2011}, $\|1_S \{ \rho^n *\ell_t- \ell_t \} \|_\infty \to 0$ as $n \to \infty$. It follows that $\|1_S \{ [ \rho^n * \ell_t^n ] - \ell_t \} \|_\infty \to 0$ as $n \to \infty$. This, together with \eqref{main_thm1:eq11} and \zcref{main_assmpt1}(5), implies
\begin{equation} \label{main_thm1:eq12a}
\sup_{x \in S} | b(t, x, \{ \rho^n * \ell_t^n \} (x), \mu^n_t) - b(t, x, \ell_t (x), \mu_t) | \xrightarrow{n \to \infty} 0.
\end{equation}

Recall that $|b| \le f_0$ and $f_0 \in \tilde L^{p_0}_{q_0}$. It follows from \eqref{main_thm1:eq11}, \eqref{main_thm1:eq12a} and DCT that
\[
\int_{\bR^d} b(t, x, \{ \rho^n * \ell_t^n \} (x), \mu^n_t) \partial_{x_i} \psi (x) \diff \mu^n_t (x)
\xrightarrow{n \to \infty}
\int_{\bR^d} b(t, x, \ell_t (x), \mu_t) \partial_{x_i} \psi (x) \diff \mu_t (x). 
\]

Taking the limit $n \to \infty$ in \eqref{main_thm1:eq10}, we get
\begin{align} 
& - \int_{\bT} \int_{\bR^d} \varphi ' (t) \psi(x) \diff \mu_t (x) \diff t \\
= & \begin{myaligned}[t]
& \sum_{i=1}^d \int_{\bT} \int_{\bR^d} b(t, x, \ell_t (x), \mu_t) \varphi (t) \partial_{x_i} \psi (x) \diff \mu_t (x) \diff t \\
& + \frac{1}{2} \sum_{i, j = 1}^d \int_{\bT} \int_{\bR^d} a^{i, j} (t, x, \mu_t) \varphi (t) \partial_{x_i} \partial_{x_j} \psi (x) \diff \mu_t (x) \diff t .
\end{myaligned} \label{main_thm1:eq13}
\end{align}

So $\ell$ satisfies the Fokker-Planck equation
\begin{equation} \label{main_thm1:eq14}
\partial_t \ell_t (x) = - \partial_{x_i} \{ b(t, x, \ell_t (x), \mu_t) \ell_t (x) \} + \frac{1}{2} \partial_{x_i} \partial_{x_j} \{ a^{i, j} (t, x, \mu_t) \ell_t (x) \} .
\end{equation}

Moreover, $\ell$ satisfies the following integrability estimate:

\begin{lemma} \label{integrability}
There exists a constant $c>0$ (depending on $\Theta_1$) such that
\[
\int_{\bT} \int_{\bR^d} \{|b(t, x, \ell_t (x), \mu_t)| + |a(t, x, \mu_t)| \} \diff \mu_t (x) \diff t \le c( 1 + \|f_0\|_{\tilde L^{p_0}_{q_0}}).
\]
\end{lemma}

\begin{proof}
By \eqref{main_thm1:eq4},
\[
\diff X_t^{n} = \bar b^n (t, X_t^{n}) \diff t + \bar \sigma^n (t, X_t^{n}) \diff B_t . 
\]

Recall that \zcref{drift_noise4} is a special case of \zcref{drift_noise3}. By \zcref{same_parameters}, all pairs $(\bar b^n, \bar \sigma^n)_{n \in \bN}$ satisfy \zcref{drift_noise3} for the same set of parameters. Then
\begin{align}
\int_{\bT} \int_{\bR^d} f_0(t, x) \diff \mu^n_t (x) \diff t &= \bE \bigg [ \int_\bT f_0 (t, X^n_t) \diff t \bigg ] \quad \text{by Tonelli's theorem} \\
&\lesssim 1 + \|f_0\|_{\tilde L^{p_0}_{q_0}} \quad \text{by \zcref{krylov_khasminskii_estimate}(2)} . \label{main_thm1:ineq14}
\end{align}

We have
\begin{align}
\int_{\bT} \int_{\bR^d} f_0(t, x) \diff \mu_t (x) \diff t &= \int_{\bT} \lim_k \int_{\bR^d} 1_{B(0, k)} (x) f_0(t, x) \ell_t (x) \diff x \diff t \label{main_thm1:eq15} \\
&\le \liminf_k \int_{\bT} \int_{\bR^d} 1_{B(0, k)} (x) f_0(t, x) \ell_t (x) \diff x \diff t \label{main_thm1:ineq15} \\
&= \liminf_k \int_{\bT} \lim_n \int_{\bR^d} 1_{B(0, k)} (x) f_0(t, x) \ell^n_t (x) \diff x \diff t \label{main_thm1:eq16} \\
&\le \liminf_k \liminf_n \int_{\bT} \int_{\bR^d} 1_{B(0, k)} (x) f_0(t, x) \ell^n_t (x) \diff x \diff t \label{main_thm1:ineq16} \\
&\le \liminf_n \int_{\bT} \int_{\bR^d} f_0(t, x) \ell^n_t (x) \diff x \diff t \\
&\lesssim 1 + \|f_0\|_{\tilde L^{p_0}_{q_0}} \quad \text{by \eqref{main_thm1:ineq14}}. \label{main_thm1:ineq17}
\end{align}

Above, \eqref{main_thm1:eq15} is due to MCT; \eqref{main_thm1:ineq15} and \eqref{main_thm1:ineq16} are due to Fatou's lemma. We will justify how \eqref{main_thm1:eq16} follows from $f_0 \in \tilde L^{p_0}_{q_0}$ and DCT:
\begin{enumerate}
\item From \eqref{main_thm1:ineq4}, we get $1_{B(0, k)} (x) f_0(t, x) \ell^n_t (x) \lesssim 1_{B(0, k)} (x) f_0(t, x)$.
\item From \eqref{main_thm1:eq12}, we get $1_{B(0, k)} (x) f_0(t, x) \ell^n_t (x) \to 1_{B(0, k)} (x) f_0(t, x) \ell_t (x)$ (as $n \to \infty$) for every $x \in \bR^d$.
\end{enumerate}

We denote by $I$ the LHS of the inequality in the statement of \zcref{integrability}. Then
\begin{align}
I &\lesssim 1 + \int_{\bT} \int_{\bR^d} f_0(t, x) \diff \mu_t (x) \diff t \\
&\lesssim 1 + \|f_0\|_{\tilde L^{p_0}_{q_0}} \quad \text{by \eqref{main_thm1:ineq17}}. 
\end{align}

This completes the proof.
\end{proof}

We have
\begin{enumerate}
\item The maps $(t, x) \mapsto b(t, x, \ell_t (x), \mu_t)$ and $(t, x) \mapsto a(t, x, \mu_t)$ are measurable.
\item By \zcref{integrability}, 
\[
\int_{\bT} \int_{\bR^d} \{|b(t, x, \ell_t (x), \mu_t)| + |a(t, x, \mu_t)| \} \diff \mu_t (x) \diff t < \infty.
\]
\item By \eqref{main_thm1:ineq12}, the map $\bT \to \sP_p (\bR^d), \,t \mapsto \mu_t$ is continuous.
\end{enumerate}

By superposition principle \cite{figalli_existence_2008,trevisan_well-posedness_2016,bogachev_ambrosiofigallitrevisan_2021} as in \cite[Section 2]{barbu2020nonlinear}, \eqref{main_eq1} has a weak solution whose marginal distribution is exactly $(\mu_t)$.

\subsection{Existence of a strong solution}

By the previous subsection, there exists a PS $(\bar \Omega, \bar {\cA},\bar \bP)$ on which there exist an $m$-BM $(\bar B_t)$, an AF $(\bar {\cF}_t)$ and a continuous $(\bar {\cF}_t)$-adapted process $(\bar X_t)$ such that
\[
\begin{cases}
\diff \bar X_t = b(t, \bar X_t, \ell_t (\bar X_t), \mu_t) \diff t + \sigma (t, \bar X_t, \mu_t) \diff \bar B_t , \\
\text{$\nu$ is the distribution of $\bar X_0$, $\mu_t$ is that of $\bar X_t$} , \\
\text{and $\ell_{t}$ is the density of $\bar X_t$} .
\end{cases}
\]

Above, the distribution of $\bar X_0$ is $\nu$, that of $\bar X_t$ is $\mu_t$, and the density of $\bar X_t$ is $\ell_{t}$. We define the map $\bar b : \bT \times \bR^d \times \sP_p (\bR^d) \to \bR^d$ by $\bar b (t, x, \varrho) \coloneq b(t, x, \ell_t (x), \varrho)$. We consider the McKean-Vlasov SDE
\begin{equation} \label{main_thm1:eq17}
\begin{cases} 
\diff Y_t = \bar b (t, Y_t, \mu'_t) \diff t + \sigma (t, Y_t, \mu'_t) \diff B_t , \\
\text{$\nu$ is the distribution of $Y_0$, and $\mu'_t$ is that of $Y_t$} .
\end{cases} 
\end{equation}

We recall that $(B_t)$ is the fixed $m$-BM on the fixed PS $(\Omega, \cA, \bP)$ introduced in \zcref{introduction}. By \cite[Theorem 1.1(1)]{huang2022singular}, \eqref{main_thm1:eq17} is weakly well-posed. On the other hand, $(\bar X_t)$ also satisfies \eqref{main_thm1:eq17}. Then, for every $t \in \bT$, we have $\mu_t = \mu'_t$ and thus the density of $Y_t$ is also $\ell_t$. In particular,
\[ 
\diff Y_t = b (t, Y_t, \ell_t (Y_t), \mu_t) \diff t + \sigma (t, Y_t, \mu_t) \diff B_t.
\]

This completes the proof.

\section{Proof of \zcref{main_thm2}} \label{main_thm2:proof}

For $k \in \{1, 2\}$, we consider the SDE
\begin{equation} \label{main_thm2:eq1}
\begin{cases} 
\diff X^k_t = b(t, X^k_t, \ell^k_t (X^k_t), \mu^k_t) \diff t + \sigma (t, X^k_t) \diff B^k_t , \\
\text{$\nu_k$ is the distribution of $X^k_0$, $\mu^k_t$ is that of $X^k_t$} , \\
\text{and $\ell^k_{t}$ is the density of $X^k_t$}.
\end{cases} 
\end{equation}

Above, $(B^k_t, t \ge 0)$ is an $m$-BM. We define measurable maps $b^k : \bT \times \bR^d \to \bR^d$ by $b^k (t, x) \coloneq b (t, x, \ell^k_t (x), \mu^k_t)$. 

\subsection{Uniqueness of marginal density}

Clearly, $(b^k, \sigma)$ satisfies \zcref{drift_noise4}. We denote $b^k_t (x) \coloneq b^k (t, x)$. By \zcref{duhamel_form1},
\[
\ell^k_t (x) = \int_{\bR^d} p_{0, t}^{0, \sigma} (y, x) \ell_{\nu_k}(y) \diff y + \int_0^t \int_{{\bR^d}} \ell^k_s (y) \langle b^k_s (y) , \nabla_y p^{0, \sigma}_{s, t} (y, x) \rangle \diff y \diff s .
\]

Then
\begin{align}
|\ell^2_t (x) - \ell^1_t (x)| &\le \begin{myaligned}[t]
& \int_{\bR^d} p_{0, t}^{0, \sigma} (y, x) | \ell_{\nu_1}(y) - \ell_{\nu_2}(y) | \diff y \\
& + \int_0^t \int_{{\bR^d}} | b^2_s (y) | \times | \ell^2_s (y) - \ell^1_s (y) | \times | \nabla_y p^{0, \sigma}_{s, t} (y, x) | \diff y \diff s  \\
& + \int_0^t \int_{{\bR^d}} \ell^1_s (y) | b^2_s (y) - b^1_s (y) | \times  | \nabla_y p^{0, \sigma}_{s, t} (y, x) | \diff y \diff s.
\end{myaligned} \label{main_thm2:ineq1}
\end{align}

We write $M_1 \preccurlyeq M_2$ if there exists a constant $c>0$ (depending on $\Theta_1$) such that $M_1 \preccurlyeq c M_2$. Thus
\begin{align}
|\ell^2_t (x) - \ell^1_t (x)| &\preccurlyeq \begin{myaligned}[t]
& \int_{\bR^d} p_{0, t}^{0, \sigma} (y, x) | \ell_{\nu_1}(y) - \ell_{\nu_2}(y) | \diff y \\
& + \int_0^t \int_{{\bR^d}} | \ell^2_s (y) - \ell^1_s (y) | \times | \nabla_y p^{0, \sigma}_{s, t} (y, x) | \diff y \diff s  \\
& + \int_0^t \int_{{\bR^d}} \ell^1_s (y) | b^2_s (y) - b^1_s (y) | \times  | \nabla_y p^{0, \sigma}_{s, t} (y, x) | \diff y \diff s.
\end{myaligned} \label{main_thm2:ineq3}
\end{align}

By \zcref{Holder_estimate2}(1),
\begin{equation} \label{main_thm2:ineq4}
\sup_{t \in \bT} \| \ell^1_t \|_{\infty} \preccurlyeq \| \ell_{\nu_1} \|_{\infty}.
\end{equation}

By \zcref{main_assmpt1}(5),
\begin{align} \label{main_thm2:ineq5}
| b^2_s (y) - b^1_s (y) | &\preccurlyeq | \ell^2_s (y) - \ell^1_s (y) | + W_p (\mu^2_s, \mu^1_s).
\end{align}

By \eqref{main_thm2:ineq3}, \eqref{main_thm2:ineq4} and \eqref{main_thm2:ineq5},
\begin{align}
|\ell^2_t (x) - \ell^1_t (x)| &\preccurlyeq \begin{myaligned}[t]
& \int_{\bR^d} p_{0, t}^{0, \sigma} (y, x) | \ell_{\nu_1}(y) - \ell_{\nu_2}(y) | \diff y \\
& + (1+\| \ell_{\nu_1} \|_{\infty}) \int_0^t \int_{{\bR^d}} | \ell^2_s (y) - \ell^1_s (y) | \times | \nabla_y p^{0, \sigma}_{s, t} (y, x) | \diff y \diff s  \\
& + \int_0^t W_p (\mu^2_s, \mu^1_s) \int_{{\bR^d}} \ell^1_s (y)  | \nabla_y p^{0, \sigma}_{s, t} (y, x) | \diff y \diff s
\end{myaligned} \\
&\eqcolon I_1 (t, x) + (1+\| \ell_{\nu_1} \|_{\infty}) I_2 (t, x) + I_3 (t, x).
\end{align}

The pair $(0, \sigma)$ satisfies \zcref{drift_noise4}. By \zcref{Holder_estimate}, there exists a constant $\lambda >0$ (depending on $\Theta_1$) such that for $i \in \{0, 1\}, 0 \le s <t \le T$ and $x, y \in \bR^d$:
\begin{equation} \label{main_thm2:ineq6}
|\nabla^i_y p^{0, \sigma}_{s, t} (y, x)| \preccurlyeq p^{i, \lambda}_{t-s} (y-x). 
\end{equation}

Then
\begin{align}
\int_{{\bR^d}} (|x|^p+1) | \nabla^i_y p^{0, \sigma}_{s, t} (y, x) | \diff x & \preccurlyeq \int_{{\bR^d}} (|x|^p+1) p^{i, \lambda}_{t-s} (y - x) \diff x \qtext{by \eqref{main_thm2:ineq6}} \\
& \preccurlyeq (t-s)^{-\frac{i}{2}} (|y|^p + 1). \label{main_thm2:ineq7}
\end{align}

We define a measurable map $f : \bT \to \bR_+$ by
\[
f(s) \coloneq \int_{\bR^d} ( |x|^p + 1) |\ell^2_s (x) - \ell^1_s (x)| \diff x .
\]

By \eqref{moment_estimate:ineq1}, $f$ is bounded. First,
\begin{align}
& \int_{\bR^d} (|x|^p+1) I_1 (t, x) \diff x \\
= & \int_0^t \int_{{\bR^d}} | \ell_{\nu_1} (y) - \ell_{\nu_2} (y) | \int_{{\bR^d}} (|x|^p+1) p^{0, \sigma}_{0, t} (y, x) \diff x \diff y \diff s \\
\preccurlyeq & \int_0^t \int_{{\bR^d}} | \ell_{\nu_1} (y) - \ell_{\nu_2} (y) | (|y|^p + 1) \diff y \diff s \qtext{by \eqref{main_thm2:ineq7}} \\
\preccurlyeq & \int_{{\bR^d}} (|y|^p + 1) | \ell_{\nu_1} (y) - \ell_{\nu_2} (y) | \diff y = f(0).
\end{align}

Second,
\begin{align}
& \int_{\bR^d} (|x|^p+1) I_2 (t, x) \diff x \\
= & \int_0^t \int_{{\bR^d}} | \ell^2_s (y) - \ell^1_s (y) | \int_{{\bR^d}} (|x|^p+1) | \nabla_y p^{0, \sigma}_{s, t} (y, x) | \diff x \diff y \diff s \\
\preccurlyeq & \int_0^t (t-s)^{-\frac{1}{2}} \int_{{\bR^d}} | \ell^2_s (y) - \ell^1_s (y) | (|y| ^p+ 1) \diff y \diff s \qtext{by \eqref{main_thm2:ineq7}} \\
= & \int_0^t (t-s)^{-\frac{1}{2}} f(s) \diff s.
\end{align}

Third,
\begin{align}
& \int_{\bR^d} (|x|^p+1) I_3 (t, x) \diff x \\
= & \int_0^t W_p (\mu^2_s, \mu^1_s) \int_{{\bR^d}} \ell^1_s (y) \int_{{\bR^d}} (|x|^p+1) | \nabla_y p^{0, \sigma}_{s, t} (y, x) | \diff x \diff y \diff s \\
\preccurlyeq & \int_0^t (t-s)^{-\frac{1}{2}} W_p (\mu^2_s, \mu^1_s) \int_{{\bR^d}} \ell^1_s (y) (|y|^p + 1) \diff y \diff s \quad \text{by \eqref{main_thm2:ineq7}} \\
\preccurlyeq & (1 + M_p (\nu_1)) \int_0^t (t-s)^{-\frac{1}{2}} W_p (\mu^2_s, \mu^1_s)  \diff s \quad \text{by \eqref{moment_estimate:ineq1}} \\
\preccurlyeq & (1+ M_p (\nu_1)) \int_0^t (t-s)^{-\frac{1}{2}} |f(s)|^{\frac{1}{p}} \diff s \quad \text{by \zcref{wp_lp}}.
\end{align}

To sum up,
\[
f(t) \preccurlyeq f(0) + (1 + \| \ell_{\nu_1} \|_{\infty} + M_p (\nu_1) ) \int_0^t (T-s)^{-\frac{1}{2}} (f(s)+|f(s)|^{\frac{1}{p}}) \diff s.
\]

Because $p=1$, we get
\[
f(t) \preccurlyeq f(0) + (1 + \| \ell_{\nu_1} \|_{\infty} + M_1 (\nu_1) ) \int_0^t (T-s)^{-\frac{1}{2}} f(s) \diff s.
\]

By Gronwall's lemma,
\begin{equation} \label{main_thm2:ineq8}
\sup_{t \in \bT} f (t) \preccurlyeq f(0) \exp \big\{2\sqrt{T}(1+\|\ell_{\nu_1}\|_{\infty}+M_1(\nu_1))\big\} .
\end{equation}

This implies the existence of the function $\Lambda$ as required in \zcref{main_thm2}(1).

\subsection{Weak and strong uniqueness of a solution}

By \eqref{main_thm2:eq1},
\begin{equation} \label{main_thm2:eq2}
\diff X^k_t = b^k(t, X^k_t) \diff t + \sigma (t, X^k_t) \diff B^k_t.
\end{equation}

Now we let $\nu \coloneq \nu_1 = \nu_2$. By \eqref{main_thm2:ineq8}, $\ell^1_t = \ell^2_t$ and $\mu^1_t = \mu^2_t$ for $t \in \bT$. Then $\pmb{b} \coloneq b^1 = b^2$. We consider the SDE
\begin{equation} \label{main_thm2:eq3}
\begin{cases} 
\diff Y_t = \pmb{b} (t, Y_t) \diff t + \sigma (t, Y_t) \diff B_t , \\
\text{$\nu$ is the distribution of $Y_0$} .
\end{cases} 
\end{equation}

By \cite[Theorem 1.1(1)]{huang2022singular}, \eqref{main_thm2:eq3} is well-posed. On the other hand, $(X^1_t)$ and $(X^2_t)$ satisfy \eqref{main_thm2:eq3}. It follows that \eqref{main_eq1} has both weak and strong uniqueness.

\section{Acknowledgment}

The author is grateful to Professor Sébastien Gadat for his generous funding during the author's PhD. The author is grateful to Professor Feng-Yu Wang for kindly addressing all questions about his work \cite{wang2023singular}. The author would like to thank Professor Michael Röckner, Dr. Zimo Hao and Professor Francesco Russo for their useful comments and suggestions.

\printbibliography

\section{Appendix} \label{appendix}

\begin{proof}[Proof of \zcref{wp_lp}]
For $B \in \cB (\bR^d)$, we denote by $\Pi(B)$ the collection of all finite measurable partitions of $B$. This means $(B_1, \dots, B_n) \in \Pi(B)$ i.f.f $\{B_1, \dots, B_n\} \subset \cB (\bR^d)$ are pairwise disjoint and $B = \bigcup_{k=1}^n B_k$. We have
\begin{align}
|\mu-\nu| (B) & = \sup \bigg \{ \sum_{k=1}^n | (\mu-\nu) (B_k) | : (B_1, \dots, B_n) \in \Pi(B)  \bigg \} \\
& = \sup \bigg \{ \sum_{k=1}^n \big | \int_{B_k} (\ell_\mu-\ell_\nu) (x) \diff x \big | : (B_1, \dots, B_n) \in \Pi(B)  \bigg \} \\
& \le \sup \bigg \{ \sum_{k=1}^n \int_{B_k} |\ell_\mu-\ell_\nu| (x) \diff x : (B_1, \dots, B_n) \in \Pi(B)  \bigg \} \\
& = \int_{B} |\ell_\mu-\ell_\nu| (x) \diff x.
\end{align}

On the other hand, we have from \cite[Proposition 7.10]{villani2003topics} that
\[
W_p^p (\mu, \nu) \le (1 \vee 2^{p-1}) \int_{\bR^d} |x|^p \diff |\mu-\nu| (x).
\]

The claim then follows.
\end{proof}

\begin{proof} [Proof of \zcref{krylov_khasminskii_estimate}]
\begin{enumerate}
\item There exists $\bar q \in (1, q)$ such that $(p, \bar q) \in \bar K$. By \cite[Theorem 3.1]{zhang2021zvonkin}, there exists a constant $c_1 > 0$ (depending on $\Theta_3, p, \bar q$) such that for $0 \le t_0 < t_1 \le T$, stopping time $\tau$ and $g \in \tilde L^{p}_{\bar q}(t_0, t_1)$:
\begin{equation} \label{krylov_khasminskii_estimate:ineq1}
\bE \bigg [ \int_{t_0 \wedge \tau}^{t_1 \wedge \tau} |g(s, X_s)|  \diff s \bigg | \mathcal F_{t_0} \bigg ] \le c_1 \|g\|_{\tilde L^{p}_{\bar q} (t_0, t_1)}.
\end{equation}

Let $\delta \coloneq \frac{1}{\bar q} - \frac{1}{q} \in (0, 1)$. By Hölder's inequality, it holds for $0 \le t_0 < t_1 \le T$ and $g \in \tilde L^{p}_{q} (t_0, t_1)$ that
\begin{equation} \label{krylov_khasminskii_estimate:ineq2}
\|g\|_{\tilde L^{p}_{\bar q} (t_0, t_1)} \le (t_1 -t_0)^{\delta} \|g\|_{\tilde L^{p}_{q} (t_0, t_1)}.
\end{equation}

We denote by $I^n_j$ the open interval $(\frac{(j-1)(t_1-t_0)}{n}, \frac{j(t_1-t_0)}{n})$ for $j=1, \ldots,n$. We fix $g \in \tilde L^{p}_{q} (t_0, t_1) \subset \tilde L^{p}_{\bar q} (t_0, t_1)$. Let $n \ge 2$ be the smallest integer such that
\begin{equation} \label{krylov_khasminskii_estimate:ineq3}
\|g\|_{\tilde L^{p}_{\bar q} (I_j^n)} \le \frac{1}{2c_1} \qtextq{for} j=1, \ldots,n.
\end{equation}

By \eqref{krylov_khasminskii_estimate:ineq1} and \cite[Lemma 3.5]{10.1214/19-AIHP959},
\begin{equation} \label{krylov_khasminskii_estimate:ineq4}
\bE \bigg [ \exp \bigg ( \int_{t_0}^{t_1} |g(s, X_s)|  \diff s \bigg ) \bigg | \mathcal F_{t_0} \bigg ] \le 2^n.
\end{equation}

By \eqref{krylov_khasminskii_estimate:ineq3}, there exists $\bar j \in \{1, \ldots, n-1\}$ such that
\begin{equation} \label{krylov_khasminskii_estimate:ineq5}
\|g\|_{\tilde L^{p}_{\bar q} (I_{\bar j}^{n-1})} > \frac{1}{2c_1}.
\end{equation}

By \eqref{krylov_khasminskii_estimate:ineq2} and \eqref{krylov_khasminskii_estimate:ineq5},
\[
\left ( \frac{t_1-t_0}{n-1} \right )^{\delta}  \|g\|_{\tilde L^{p}_{q} (I_{\bar j}^{n-1})} > \frac{1}{2c_1} .
\]

Then
\[
n < 1 + T(2c_1)^{-\frac{1}{\delta}} \|g\|_{\tilde L^{p}_{q} (t_0, t_1)}^{1/\delta}.
\]

The estimate \eqref{krylov_khasminskii_estimate:ineq0a} then follows with $k \coloneq \frac{1}{\delta}$.

\item We follow an elegant idea in \cite[Lemma 2.3]{huang2022singular}. Let $C_j \coloneq e^{j-1}$. We define $h : \bR_+ \to \bR_+$ by $h(r) \coloneq |\ln (C_j +r)|^j$. Then $h$ is concave. We have
\begin{align}
& \bE \bigg [ \left ( \int_{t_0}^{t_1} |g(s, X_s)| \diff s \right )^j \bigg | \mathcal F_{t_0} \bigg ] \\
\le  & \bE \bigg [ \bigg \{ \ln \bigg (C_j + \exp \bigg ( \int_{t_0}^{t_1} |g(s, X_s)| \diff s \bigg ) \bigg ) \bigg \}^j \bigg | \mathcal F_{t_0} \bigg ] \\
\le  &  \bigg \{ \ln \bigg (C_j + \bE \bigg [ \exp \bigg ( \int_{t_0}^{t_1} |g(s, X_s)| \diff s \bigg ) \bigg | \mathcal F_{t_0} \bigg ]  \bigg ) \bigg \}^j \qtext{by Jensen's inequality} \\
\le  &  \{ \ln [ C_j + \exp(c(1+ \|g\|_{\tilde L^{p}_{q} (t_0, t_1)}^k ) ) ] \}^j .
\end{align}

Above, the constants $c, k >0$ are given by \eqref{krylov_khasminskii_estimate:ineq0a}. As a result, there exists a constant $\bar C_j >0$ (depending on $c, j$) such that
\[
\bE \bigg [ \left ( \int_{t_0}^{t_1} |g(s, X_s)| \diff s \right )^j \bigg | \mathcal F_{t_0} \bigg ] \le \bar C_j (1 + \|g\|_{\tilde L^{p}_{q} (t_0, t_1)}^k)^j.
\]

Replacing $g$ with $\frac{g}{\|g\|_{\tilde L^{p}_{q} (t_0, t_1)}}$ in above inequality, we obtain
\[
\bE \bigg [ \left ( \int_{t_0}^{t_1} |g(s, X_s)| \diff s \right )^j \bigg | \mathcal F_{t_0} \bigg ] \le \bar C_j 2^j \|g\|_{\tilde L^{p}_{q} (t_0, t_1)}^j.
\]

The estimate \eqref{krylov_khasminskii_estimate:ineq0b} then follows. This completes the proof.
\end{enumerate}
\end{proof}

\begin{proof}[Proof of \zcref{duhamel_form1}]

First, we recall some notions related to \eqref{eq1}. The semigroup $(P_{s, t}^{b, \sigma})_{0 \le s < t \le T}$ is defined for $x \in {\bR}^d$ and bounded measurable function $f:\bR^d \to \bR$ by
\begin{equation} \label{prelim:sde:semi_group1}
P_{s, t}^{b, \sigma} f (x) \coloneq \bE [f(X^x_{s, t})] = \int_{\bR^d} p^{b, \sigma}_{s, t} (x, y) f(y) \diff y.
\end{equation}

The differential operator $(L_t^{b, \sigma})_{t \in \bT}$ is defined for $f \in C^2 (\bR^d)$ and $x \in \bR^d$ by
\[
L_t^{b, \sigma} f (x) \coloneq \langle b_t (x), \nabla f (x) \rangle + \frac{1}{2} \trace ( a_t (x) \nabla^2 f (x) ) .
\]

The backward Kolmogorov equation holds, i.e., for $f \in C^2_b (\bR^d), x \in \bR^d$ and $0 \le s < t \le T$:
\begin{equation} \label{backward_kolmogorov}
\partial_s P_{s, t}^{b, \sigma} f (x) + L_s^{b, \sigma} P_{s, t}^{b, \sigma} f (x) =0.
\end{equation}

Next we go on to prove the result. Let $f \in C^\infty_c (\bR^d)$. Applying Itô's lemma on $[0, t] \times {\bR}^d \to \bR, \, (s, x) \mapsto (P_{s, t}^{0, \sigma} f) (x)$, we have
\begin{align}
\diff \{ (P_{s, t}^{0, \sigma} f) (X_s) \} & =  \{(\partial_s + L_s^{b, \sigma}) (P_{s, t}^{v, \sigma} f)\} (X_s) \diff s + \diff  M_s \\
& =  \{(\partial_s +L_s^{0, \sigma}) + (L_s^{b, \sigma} - L_s^{0, \sigma}) \} (P_{s, t}^{0, \sigma} f) (X_s) \diff s + \diff  M_s\\
& = \{ (L_s^{b, \sigma} - L_s^{0, \sigma}) (P_{s, t}^{0, \sigma} f) \} (X_s) \diff s + \diff  M_s. \label{duhamel_form1:2:eq1}
\end{align}

Above, $M_0=0$ and $\diff M_s = \{ \nabla (P_{s, t}^{0, \sigma} f)  (X_s) \}^\top \sigma (s, X_s) \diff B_s$ for $s \in [0, t]$; and \eqref{duhamel_form1:2:eq1} is due to backward Kolmogorov equation \eqref{backward_kolmogorov}. Then
\begin{equation} \label{duhamel_form1:2:eq2}
f (X_t) = \begin{multlined}[t]
(P_{0, t}^{0, \sigma} f) (X_0) + \int_0^t \langle b_s (X_s) , \nabla (P_{s, t}^{0, \sigma} f)  (X_s) \rangle \diff s \\
+ \int_0^t \{ \nabla (P_{s, t}^{0, \sigma} f)  (X_s) \}^\top \sigma (s, X_s) \diff B_s .
\end{multlined}
\end{equation}

Clearly, $f$ and thus $P_{0, t}^{0, \sigma} f$ are bounded. Let's prove that $\| \nabla (P_{s, t}^{0, \sigma} f) \|_\infty < \infty$. It suffices to verify that $P_{s, t}^{0, \sigma} f$ is Lipschitz. We consider the SDE
\[
\diff Y^{x}_{s,t} = v (t, Y^{x}_{s, t}) \diff t + \sigma (t, Y^{x}_{s, t}) \diff B_s, \quad t \in [s, T], Y^{x}_{s, s}=x.
\]

By \cite[Theorem 1.1(2)]{huang2022singular}, there exists a constant $c_1 >0$ such that
\begin{align}
|P_{s, t}^{0, \sigma} f (x) - P_{s, t}^{0, \sigma} f (y) | &= |\bE [f(Y^{x}_{s, t})] - \bE [f(Y^{y}_{s, t})]| \\
&\le \| \nabla f \|_{\infty} \bE [ |Y^{x}_{s, t} - Y^{y}_{s, t}| ] \\
&\le c_1 \| \nabla f \|_{\infty} |x-y| .
\end{align}

We have
\begin{align}
\bE \bigg [ \int_0^t | b_s (X_s) | \times |\nabla (P_{s, t}^{0, \sigma} f)  (X_s) | \diff s \bigg ] & \le \| \nabla (P_{s, t}^{0, \sigma} f) \|_\infty \bE \bigg [ \int_0^t g (t, X_s) \diff s \bigg ] \\
& \lesssim \| \nabla (P_{s, t}^{0, \sigma} f) \|_\infty \| g \|_{\tilde L^{\bar p}_{\bar q} (t)} \qtext{by \zcref{krylov_khasminskii_estimate}(2)}.
\end{align}

So each term in \eqref{duhamel_form1:2:eq2} is $\bP$-integrable. Then
\begin{align} \label{duhamel_form1:2:eq3}
\int_{\bR^d} \ell_t (x) f(x) \diff x &= \begin{multlined}[t]
\int_{\bR^d} \ell_{\nu} (x) (P_{0, t}^{0, \sigma} f)(x) \diff x \\
+ \int_0^t \int_{\bR^d} \ell_s (x) \langle b_s (x) , \nabla (P_{s, t}^{0, \sigma} f)  (x) \rangle \diff x \diff s.
\end{multlined}
\end{align}

By \zcref{Holder_estimate}(1) and Leibniz integral rule,
\[
\nabla (P_{s, t}^{0, \sigma} f)  (x) = \nabla_x \int_{\bR^d} p_{s, t}^{0, \sigma} (x, y) f(y) \diff y = \int_{\bR^d} \nabla_x p_{s, t}^{0, \sigma} (x, y) f(y) \diff y.
\]

So \eqref{duhamel_form1:2:eq3} is equivalent to
\begin{equation}
\int_{\bR^d} \ell_t (x) f(x) \diff x = \begin{multlined}[t]
\int_{\bR^d} \left ( \int_{\bR^d} p_{0, t}^{0, \sigma} (y, x) \ell_{\nu}(y) \diff y \right ) f (x) \diff x \\
+ \int_{\bR^d} \left (\int_0^t \int_{\bR^d} \ell_s (y) \langle b_s (y) , \nabla_y p_{s, t}^{0, \sigma} (y, x) \rangle \diff y \diff s \right ) f (x) \diff x .
\end{multlined}
\end{equation}

The required representation then follows.
\end{proof}

\begin{proof}[Proof of \zcref{tail_moment}]
WLOG, we consider $R > 0$. We write $M_1 \lesssim M_2$ if there exists a constant $c>0$ depending on $(\Theta_4, \nu, p)$ such that $M_1 \le c M_2$. We denote by $\ell_t$ the density of $X_t$. Then
\[
\ell_t (x) = \int_{\bR^d} p^{b, \sigma}_{0, t} (y, x) \diff \nu (y) .
\]

By \zcref{Holder_estimate}(1), there exists a constant $\lambda \in (0, 1)$ depending on $\Theta_4$ such that $p^{b, \sigma}_{0, t} (y, x) \lesssim p^{0, \lambda}_t (y-x)$. Then $\ell_t (x) \lesssim p^{0, \lambda}_t * \nu (x)$. Let $Z$ be a standard normal random variable on $\bR^d$. Let $Y$ be a random variable on $\bR^d$, independent of $Z$, with distribution $\nu$. Let $c_t \coloneq \sqrt{\frac{t}{2 \lambda}}$ and $s \coloneq \frac{1}{2}$. Then
\begin{align}
\int_{B_R^c} | \cdot |^p \diff \mu_t & \lesssim \bE [ 1_{\{ | c_t Z + Y | > R \}} | c_t Z + Y |^p ] \\
& \lesssim \bE [ ( 1_{\{ | c_t Z | > s R \}} + 1_{\{ | Y | > (1-s) R \}} ) (| c_t|^p |Z|^p + |Y |^p) ] \\
& \lesssim \begin{myaligned}[t]
& \bE [ 1_{\{ | Z | > \frac{s R}{c_T} \}} |Z|^p ] + \bE [ 1_{\{ | Z | > \frac{s R}{c_T} \}} |Y|^p ] \\
& + \bE [  1_{\{ | Y | > (1-s) R \}} |Z|^p ] + \bE [ 1_{\{ | Y | > (1-s) R \}} |Y |^p ]
\end{myaligned} \\
& \eqcolon I_1 (R) + I_2 (R) + I_3 (R) + I_4 (R) \\
& \eqcolon \phi (R).
\end{align}

By Markov's inequality,
\begin{align}
\bP [ | Z | > s R /c_T ] & \le \frac{c_T \bE [ |Z| ]}{sR} , \\
\bP [ | Y | > (1-s) R ] & \le \frac{\bE [ |Y| ]}{(1-s) R} .
\end{align}

We have $\bE [|Z|^p ] + \bE |Y|^p ] < \infty$. By dominated convergence theorem (DCT),
\[
\lim_{R \to \infty} I_1 (R) = \lim_{R \to \infty} I_2 (R) = \lim_{R \to \infty} I_3 (R) = \lim_{R \to \infty} I_4 (R) = 0.
\]

This completes the proof.
\end{proof}

\end{document}